\documentclass[reqno,12pt,20pt]{amsart}
\usepackage[ansinew]{inputenc}
\usepackage[english]{babel}
\usepackage[all]{xy}
\usepackage{amsmath,latexsym,amssymb,verbatim}
\input arrow.tex

\newcommand{\OO}{{\mathcal O}}

\newcommand{\LL}{{\mathbb L}}
\newcommand{\FF}{{\mathtt F}}

\newcommand{\Lc}{{\mathcal L}}

\newcommand{\ZZ}{{\mathbb Z}}

\newcommand{\PP}{\mathbb P}
\newcommand{\RR}{\mathbb R}
\newcommand{\0}{\bold 0}
\newcommand{\un}{\bold 1}

\newcommand{\depth}{\operatorname{depth}}

\newcommand{\enumera}{\begin{enumerate}}
\newcommand{\eenumera}{\end{enumerate}}

\newcommand{\A}{{\mathbb A}}

\newcommand{\stella}{{\scriptscriptstyle\bigstar}}
 
\DeclareMathOperator{\Hom}{{Hom}}

\newcommand{\alineas}[1]{\begin{array}{#1}}
\newcommand{\alinea}{\begin{array}{l}}
\newcommand{\ealinea}{\end{array}}
\newcommand{\ealineas}{\end{array}}

\newcommand{\pun}{{\scriptscriptstyle \bullet}}

\newcommand{\HHom}{\mathcal{H}om}

\newcommand{\HExt}{\mathcal{E}xt}

\newcommand{\Qcoh}{\operatorname{Qcoh}}
\newcommand{\Perf}{\operatorname{Perf}}

\newcommand{\Shv}{\operatorname{Shv}}

\newcommand{\supp}{\operatorname{supp}}
\newcommand{\dsupp}{\operatorname{d-supp}}

\newcommand{\D}{\text{D}}

\DeclareMathOperator{\Ext}{{Ext}}

\DeclareMathOperator{\id}{{id}}

\theoremstyle{plain}
\newtheorem{thm}{Theorem}[section]
\newtheorem{lem}[thm]{Lemma}
\newtheorem{cor}[thm]{Corollary}
\newtheorem{prop}[thm]{Proposition}
\newtheorem{defn}[thm]{Definition}
\newtheorem{rem}[thm]{Remark}
\newtheorem{rems}[thm]{Remarks}

\newtheorem*{ex}{Example}

\newtheorem{cosa}[thm]{}

\numberwithin{equation}{thm}

\begin{document}

\title{Dualizing and canonical complexes on finite posets}

\author{Fernando Sancho de Salas and Alejandro Torres Sancho}

\address{ Departamento de
Matem\'aticas and Instituto Universitario de F\'isica Fundamental y Matem\'aticas (IUFFyM)\newline
Universidad de Salamanca\newline  Plaza de la Merced 1-4\\
37008 Salamanca\newline  Spain}
\email{fsancho@usal.es}
\email{atorressancho@usal.es}



\thanks {Work supported by Grant PID2021-128665NB-I00 funded by MCIN/AEI/ 10.13039/501100011033 and, as appropriate, by ``ERDF A way of making Europe''.}
 
\begin{abstract}   We develop Grothendieck's theory of dualizing complexes on finite posets, and its subsequent theory of Cohen-Macaulayness.
\end{abstract}



\maketitle

\section*{Introduction} 

A. Grothendieck developed the theory of dualizing complexes and its relation with Cohen--Macaulayness or Gorensteinness in the framework of schemes and quasi-coherent sheaves on them (\cite{Hartshorne}). Here we adopt his point of view for finite topological spaces or posets, where the category of sheaves of abelian groups plays the role of that of quasi-coherent modules. This was already initiated in \cite{ST}, where we developed the basic facts of dualizing complexes, canonical sheaf and Cohen--Macaulayness on the finite space $\A^n_{\FF_1}:=\{\text{subsets of } \Delta_n\}$, where $\Delta_n$ is a set with $n$ elements, which is the ambient space where simplical complexes live and it is the underlying topological space of the $n$-dimensional affine  Deitmar's scheme over $\FF_1$ (\cite{Deitmar}, \cite{Thas}). Since the underlying topological space of a Dietmar $\FF_1$-scheme of finite type is a finite topological space, this paper may be viewed as the development of Grothendieck's theory of dualizing complexes on such schemes.

Let $X$ be a finite topological space,  $D(X)$ the derived category of complexes of sheaves of abelian groups on $X$ and $D^b_c(X)$ (resp. $D_c(X)$) the subcategory of complexes with bounded and finitely generated cohomology (resp. finitely generated cohomology). There are two different notions of ``dualizing complex'' on $X$ which have to be clearly distinguished. 

On the one side, one has the {\em global dualizing complex} $D_X$, introduced in \cite{Na}, which is the complex that dualizes cohomology, that is, it satisfies Grothendieck-Verdier duality:
\[ \RR\Hom_X^\pun(F,D_X)=\RR\Hom_\ZZ^\pun(\RR\Gamma(X,F),\ZZ) \] for any $F\in D(X)$. If $X$ is a local space (i.e., it has a unique closed point, denoted by $\0$), we introduce a ``local dualizing complex'' $D_X^\0$, which dualizes local cohomology
\[ \RR\Hom_X^\pun(F,D_X^\0)=\RR\Hom_\ZZ^\pun(\RR\Gamma_\0(X,F),\ZZ). \] These complexes, $D_X$ and $D_X^\0$, always exist and are unique up to isomorphisms.

On the other side one has the notion of a ``dualizing complex'' $\Omega\in D^b_c(X)$ obtained from reflexivity. These complexes, named dualizing complexes by Grothendieck in the framework of schemes, will be called here {\em canonical complexes} for two reasons. Firstly,   to avoid confusion with the global and local dualizing complexes defined above. Secondly, the term canonical complex or canonical sheaf has become quite standard to mention dualizing complexes on noetherian rings. Thus, a canonical complex is a complex $\Omega\in D^b_c(X)$ such that the natural morphism\vskip 6pt

\centerline{ $ F\longrightarrow  \RR\HHom_X^\pun(\RR\HHom_X^\pun(F,\Omega),\Omega)$ } \vskip 6pt
\noindent is an isomorphism  for any $F\in D_c(X)$,  i.e., a complex that induces {\em reflexivity} on $D_c(X)$. These canonical complexes will be the main protagonist in this paper. In contrast with the global and local dualizing complexes,  a canonical complex does not necessarily exist, and, if it exists, it is unique only up to a shift and tensoring with an invertible sheaf, as it happens on schemes. However, as we shall see, these canonical complexes are closely related to global and local dualizing complexes (again, as it happens on schemes).

In section \ref{Local-global-section} we recall the notion of relative dualizing complex $D_{X/S}$ of a continous map $f\colon X\to S$ between finite spaces and introduce the notion of {\em local relative dualizing complex $D_{X/S}^Y$ along a closed subset $Y$ of $X$}. We see in Theorem \ref{loc-rel-dualizing} how these two complexes are related and in Theorem \ref{closed-loc-rel-dualizing} how the local relative dualizing complex behaves with respect to closed subspaces. Then, we shall focus our attention in the case that $S$ is a point and $Y$ is the closed point  of a local space $X$; this is the local dualizing complex $D_X^\0$ mentioned above, and will play an essential role. In Proposition \ref{restricciondualizantelocal} we see the relation between the local dualizing complexes and the sheaves $\ZZ_{\{x\}}=$ ``constant sheaf supported at $x$''. These sheaves will appear and play an important role throughout the paper. They are the analog of the residual fields of points in schemes.

Section \ref{canonicalcomplex-section} is devoted to canonical complexes and has a great parallelism with \cite{Hartshorne}. As we have already mentioned, such a complex may not exist (if it does, we say that $X$ is {\em dualizable}). In case it exists, it is unique up to a shift and tensoring with an invertible sheaf (Theorem \ref{unicity}). Its existence imposes some topological conditions on $X$ (Theorem \ref{top-conditions}): (1) A codimension function on $X$ exists, in particular  $X$ is catenary. (2) every interval has the cohomology of a sphere (of the same dimension than the interval). We shall also see that these conditons are sufficient for the existence of a canonical complex when $X$ is either local or irreducible (Theorems   \ref{existenciadualizante1}, \ref{existenciadualizante2}). Thus, we shall deduce that any simplicial complex or any locally closed subspace of a simplical complex  admits a canonical complex. Moreover, the canonical complex of a simplicial complex $K$ agrees, via the correspondence constructed in \cite{ST}, with the dualizing complex of the scheme $S_K$ associated to $K$ by the Stanley--Reisner correspondence. Theorem  \ref{thm1} states the main cohomological properties of a canonical complex and its relation to local cohomology. A key consequence is that on a local space $X$ a canonical complex coincides (up to a shift) with the local dualizing complex and its restriction to $X^*:=X\negmedspace-\negmedspace\0$ coincides with the global dualizing complex $D_{X^*}$. Theorem \ref{loc-duality-thm} summarizes the rich cohomological properties of a local and dualizable space, and it has a remarkable resemblance to the analogous results on a local noetherian ring with a dualizing complex.

Section \ref{CM-section} is devoted to the development of Cohen--Macaulayness theory on finite spaces from that of canonical complexes, mimicking Grothendieck's point of vue. We  first introduce the notions of depth and dimension of a sheaf and define the Cohen--Macaulayness of a sheaf on a local and dualizable space $X$ in terms of its dual with the canonical complex $D_X^\0$ (Definition \ref{CM-def}). After giving the generic structure of a Cohen--Macaualy sheaf (Proposition \ref{generic-CM}), we fully characterize Cohen--Macaulay sheaves in two important cases: when $F$ is generically torsion free on the one side and when $F$ is a torsion sheaf on the other side. This is done in Theorems \ref{CM-freesheaf} and \ref{CM-torsionsheaf}.  We shall then define a Cohen--Macaulay local space $X$ as a local space such that the constant sheaf $\ZZ$ is Cohen--Macaulay. This means that the local dualizing complex is a sheaf (shifted by the dimension of $X$). That is, a Cohen--Macaulay local space is a local space with a {\em canonical sheaf} $\omega_X$. The characterization of Cohen--Macaulay local spaces in terms of purity and local cohomology groups is given in Theorem \ref{CM-space} (see also Corollary \ref{CM-ateverypoint}). In Theorem \ref{CM-closed} we prove that a closed subset $K$ of a local space is Cohen--Macaulay if and only if the sheaf $\ZZ_K$ (constant sheaf supported on $K$) is Cohen--Macaulay. In particular, a closed subset $K$ of a Cohen--Macaulay space $X$ is Cohen--Macaulay if and only if all $\HExt^i_X(\ZZ_K,\omega_X)$ except one vanish: it is the {\em same} criterion of Cohen--Macaulayness of a closed subscheme of a Cohen-Macaulay scheme. In the simplicial case, the canonical sheaf of a Cohen--Macaulay simplicial complex $K$ agrees, via the correspondence constructed in \cite{ST}, with the canonical module of the scheme $S_K$ associated to $K$.  All the cohomological properties of the canonical sheaf of a Cohen--Macaulay local space are summarized in the following theorem (Theorem \ref{omega1}), that shows the similarity with the properties of the canonical sheaf of a  Cohen--Macaulay local ring:

\medskip
\noindent{\bf Theorem}. {\em Let $X$ be an $n$-dimensional Cohen--Macaulay local   finite space with canonical sheaf $\omega_X$. Let  $\0$ be the closed point of $X$ and $X^*=X\negmedspace-\negmedspace\0$.  Then:  \medskip

{\rm\bf (A)} For any $F\in D(X )$ one has  short exact sequences  
\[\aligned 0\to \Ext^1_\ZZ(H^{n+1-i}_\0(X, F),\ZZ  )  \to\Ext^i_{X}(F ,\omega_{X}) \to \Hom_\ZZ(H^{n-i}_\0(X, F),\ZZ)\to 0
\\ 0\to \Ext^1_\ZZ(\Ext_{X}^{n+1-i} (F ,\omega_{X}),\ZZ  )  \to H^i_\0(X,F) \to \Hom_\ZZ(\Ext_{X}^{n-i} (F ,\omega_X ),\ZZ)\to 0.\endaligned\]
Taking $F=\ZZ_K$, for some closed subset $K$ of $X$, we obtain   exact sequences
\[\aligned  0\to \Ext^1_\ZZ(H^{n+1-i}_\0(K, \ZZ),\ZZ  )  \to H^i_{K}(X,\omega_{X}) \to \Hom_\ZZ(H^{n-i}_\0(K, \ZZ),\ZZ)\to 0
\\ 0\to \Ext^1_\ZZ(H^{n+1-i}_{K}(X,\omega_{X}),\ZZ  )  \to H^i_{\0}(K,\ZZ) \to \Hom_\ZZ(H^{n-i}_{K}(X,\omega_{X}),\ZZ)\to 0.
\endaligned \]
 
{\rm\bf (B)} The sheaf $$\omega_{X^*}:={(\omega_X)}_{\vert X^*}$$ is a global dualizing sheaf on $X^*$, i.e., $\omega_{X^*}[n-1]$ is the global dualizing complex of $X^*$. Thus, for any $F\in D(X^*)$ one has   short exact sequences (notice that $n-1=\dim X^*$)
\[\aligned 0\to \Ext^1_\ZZ(H^{n-i}(X^*, F),\ZZ  )  \to\Ext^i_{X^*}(F ,\omega_{X^*}) \to \Hom_\ZZ(H^{n-1-i}(X^*, F),\ZZ)\to 0
\\ 0\to \Ext^1_\ZZ(\Ext^{n-i}_{X^*}(F ,\omega_{X^*}),\ZZ  )  \to H^{i}(X^*, F) \to \Hom_\ZZ(\Ext^{n-1-i}_{X^*}(F ,\omega_{X^*}),\ZZ)\to 0.
\endaligned \] In particular, taking $F=\ZZ$, one has   exact sequences
\[\aligned 0\to \Ext^1_\ZZ(H^{n-i}(X^*, \ZZ),\ZZ  )  \to H^i (X^* ,\omega_{X^*}) \to \Hom_\ZZ(H^{n-1-i}(X^*, \ZZ),\ZZ)\to 0
\\ 0\to \Ext^1_\ZZ(H^{n-i}(X^* ,\omega_{X^*}),\ZZ  )  \to H^i (X^* ,\ZZ) \to \Hom_\ZZ(H^{n-1-i}(X^* ,\omega_{X^*}),\ZZ)\to 0,
\endaligned \] and for any  closed subset $K^*$ of $X^*$, taking $F=\ZZ_{K^*}$,   exact sequences
\[ \aligned 0\to \Ext^1_\ZZ(H^{n-i}(K^*, \ZZ),\ZZ  )  \to H^i_{K^*} (X^* ,\omega_{X^*}) \to \Hom_\ZZ(H^{n-1-i}(K^*, \ZZ),\ZZ)\to 0
\\ 0\to \Ext^1_\ZZ(H^{n-i}_{K^*} (X^* ,\omega_{X^*}),\ZZ  )  \to H^i  (K^* ,\ZZ) \to \Hom_\ZZ(H^{n-1-i}_{K^*} (X^* ,\omega_{X^*}),\ZZ)\to 0.
\endaligned  \]

{\rm\bf (C)} If $K\subseteq X$ is a Cohen--Macaulay closed subset of codimension $d$, one has Gysin  isomorphisms
\[ \aligned H^i_K(X,\omega_X)& = H^{i-d}(K,\omega_K) =\left\{\aligned 0\qquad, &\quad \text{ for }i\neq d
\\   H^{n-d}_\0(K,\ZZ)^*, &\quad \text{ for } i=d, \endaligned\right.
\\ 
\\ H^i_{K^*}(X^*,\omega_{X^*})&=H^{i-d}(K^*,\omega_{K^*})=\left\{ \aligned 0\qquad, & \quad\text{ for } i\neq d, n-1 
\\ \ZZ\qquad, & \quad\text{ for } i=n-1
\\ H^{n-d-1}(K^*,\ZZ)^*, & \quad\text{ for } i=d\endaligned\right.
\endaligned \] }
\medskip


 The notion of a Cohen--Macaulay sheaf or space in the non local case is reduced to the local case. The last results of section \ref{CM-section} are devoted to study the behaviour of Cohen--Macaulayness under products and barycentric subdivision (Theorems \ref{CM-product} and \ref{barycentric-CM}). We end our study of Cohen--Macaulayness by comparing it with the classic one on posets, as developed in  \cite{B}.

Section \ref{prelim-section} is devoted to fix some basic definitions, notations and results on finite spaces and their derived categories. A few technical results are proved to be used later. Finally, section \ref{section1} is mainly devoted to see the basic interplay between the sheaves $\ZZ_{\{x\}}$ and the cohomology of intervals, which will be widely used in the rest of the paper.

\section{Preliminaries and notations}\label{prelim-section}

\subsection{Finite spaces} Throughout this paper, a finite space means a finite and $T_0$ topological space. As it is well known, this is equivalent to a finite poset, but we shall preferably take the topological point of view. For each $x\in X$ we shall denote
\[ \aligned U_x &=\text{ smallest open subset of } X \text{ containing } x, 
\\ C_x & = \text{ smallest closed subset of } X \text{ containing } x =\text{ closure of } x.\endaligned\]
and we shall also denote $U_x^*:=U_x\negmedspace-\negmedspace\{x\}$, $C_x^*:=C_x\negmedspace-\negmedspace\{x\}$.

The partial order on $X$ is given by
\[x\leq y \Leftrightarrow C_x\subseteq C_y \Leftrightarrow U_x\supseteq U_y.\] A subset $U$ of $X$ is open if and only if it is increasing: $x\in U$ and $y\geq x$ implies $y\in U$. Analogously, a subset $C$ of $X$ is closed if and only if it is decreasing: $x\in C$ and $y\leq x$ implies $y\in C$. One has:
\[ \aligned U_x &=\left\{ p\in X: p\geq x\right\}  
\\ C_x &=\left\{ p\in X: p\leq x\right\}.\endaligned\] A point $x\in X$ is closed if and only if it is minimal; analogously, $x$ is open if and only if it is maximal. Open points will be called {\em generic points} of $X$, since their closures are the irreducible components of $X$.

A map $f\colon X\to Y$ between finite spaces is continuous if and only if it is monotone: $x\leq x'\Rightarrow f(x)\leq f(x')$. 

The dual space of $X$ is denoted by $X^{\text{\rm op}}$. It is the same underlying set with the dual topology (i.e., the reversed partial order): an open subset of $X^{\text{\rm op}}$ is a closed subset of $X$.

\begin{defn} {\rm We say that a finite space $X$ is {\em local} if it has a unique closed point (i.e., $X$ has a minimum) which we shall usually denote by $\0$. The open complement $X-\0$ will be usually denoted by $X^*$. For any point $x\in X$, $U_x$ is local, with closed point $x$, and $C_x$ is irreducible with generic point $x$.}
\end{defn}

\begin{defn}{\rm  The {\em dimension}  of a finite space $X$, denoted by $\dim X$, is the maximal length of the chains  
\[ C_0 \subset C_1\subset \cdots\subset C_n\] of irreducible closed subsets of $X$; equivalently, it is the maximal length of the chains $x_0 <x_1 <\cdots < x_n$ of points of $X$.

A finite space  $X$ is called {\em pure} if all the maximal chains $x_0 <x_1 <\cdots < x_n$ have length equal to $\dim X$.}\end{defn}

\subsubsection{Simplicial complexes} Let us denote $\Delta_n=\{ 1,\dots, n\}$ the set of $n$ elements and
\[ \A^n_{\FF_1}:=\{\text{subsets of }\Delta_n\}\] which is a finite poset with the preorder given by inclusion: $p\leq q\Leftrightarrow p\subseteq q$.

$\A^n_{\FF_1}$ is a local and irreducible space whose closed point $\0$ represents the empty subset of $\Delta_n$ and whose generic point is the total subset of $\Delta_n$. We shall denote  $$\PP^{n-1}_{\FF_1}:=\A^n_{\FF_1}-\0$$ which is an open subset of $\A^n_{\FF_1}$. 

\begin{defn}{\rm A {\em simplicial complex} $K$ (with at most $n$ vertices) is a closed subset of $\A^n_{\FF_1}$. Notice that the empty subset of $\Delta_n$ is admitted as a simplex of $K$. If we do not want to admit the empty subset as a simplex, then we define: a {\em projective simplicial complex} is a closed subset of $\PP^{n-1}_{\FF_1}$. Of course, if $K$ is a simplical complex, then $K^*=K-\0$ is a projective simplicial complex, and conversely, adding $\0$ to a projective simplicial complex yields a simplical complex. As we shall see, simplicial complexes have a behaviour similar to affine algebraic schemes (closed subschemes of the affine $n$-dimensional scheme) and projective simplical complexes a similar one to projective schemes. It is important to warn the reader that projective simplicial complexes are usually called {\em abstract simplicial complexes} in the combinatorics literature.

Finally, we say that a finite space $X$ is {\em locally simplicial} if $U_x$ is a simplicial complex for any $x\in X$. Any simplicial complex and any projective simplicial complex is locally simplicial. A locally simplical finite space will have a behaviour similar to a scheme  of finite type.}
\end{defn}

\begin{ex} {\rm For any finite space $X$, its {\em barycentric subdivision} $\beta X$ is the set of (non empty) totally ordered subsets of $X$. It is a closed subset of the poset of all (non-empty) subsets of $X$. Hence, $\beta X$ is a projective simplicial complex.}
\end{ex}

\subsection{Sheaves} As in any topological space, one has sheaves of abelian groups on a finite space $X$. We shall denote by $\Shv(X)$ the category of  sheaves of abelian groups on $X$. As it is well known, on a finite space  a sheaf becomes a quite simpler notion than in general topological spaces: a sheaf $F$ of abelian groups on a finite space $X$ is equivalent to a functor from  $X^{\text{\rm op}}$  to abelian groups. We shall follow the notations of \cite{Sanchoetal}.

A sheaf $F$ is {\em finitely generated} if $F_x$ is a finitely generated abelian group for any $x\in X$. For any sheaf $F$ on $X$ and any locally closed subset $S$ of $X$, we shall denote by $F_S$ the sheaf $F$ supported on $S$; it is the unique sheaf satisfying
\[ (F_S)_{\vert S}=F_{\vert S}\quad , \quad (F_S)_{\vert X-S}=0.\] 
To avoid confusions, for each $x\in X$, $F_x$ denotes the stalk of $F$ at $x$ and $F_{\{x\}}$ denotes the sheaf $F$ supported on $x$.

We shall denote by $D(\ZZ)$ the derived category of the category $C(\ZZ)$ of complexes of abelian groups, $D_c(\ZZ)$ the full subcategory of complexes with finitely generated cohomology groups and $D^b (\ZZ)$ the full subcategory of complexes with bounded cohomology. Finally, $D^b_c(\ZZ) = D^b (\ZZ) \cap D_c(\ZZ)$. Notice that any $E\in D^b_c(\ZZ)$ is perfect, since $\ZZ$ is regular.
For any $E\in D(\ZZ)$, we denote $E^\vee:=\RR\Hom^\pun_\ZZ(E,\ZZ)$. The functor $(\quad)^\vee$ preserves $D_c(\ZZ)$ and $D^b(\ZZ)$. The natural morphism $E\to E^{\vee\vee}$ is an isomorphism for any $E\in D_c(\ZZ)$. We shall use the following elementary result:

\begin{lem}\label{lem0} Let $K\in D^b_c(\ZZ)$. The following conditions are equivalent
\begin{enumerate}\item $K\simeq\ZZ [r]$ for some integer $r$.
\item There exists $K'\in D^b_c(\ZZ)$ such that $K'\overset\LL\otimes K \simeq \ZZ$.
\item $\ZZ\to \RR\Hom_\ZZ^\pun (K,K) $ is an isomorphism.

\end{enumerate}
\end{lem}

We shall denote by $C(X)$ the category of complexes of sheaves of abelian groups on $X$ and $D(X)$ its derived category. Then,  $D_c(X)$ denotes the full subcategory  of complexes $F$ such that  the sheaves $H^i(F)$ are finitely generated, $D^b(X)$ the full subcategory of complexes with bounded cohomology and $D^b_c(X)=D_c(X)\cap D^b(X)$. An object $F\in D^b_c(X)$ is sometimes called {\em bounded and coherent}. These subcategories are detected fibrewise, that is, a complex $F\in D(X)$ belongs to $D_c(X)$ (resp. $D^b(X)$) if and only if $F_x\in D_c(\ZZ)$ (resp. $D^b(\ZZ)$) for any $x\in X$.

For any complexes $F,G\in D(X)$, we shall denote $\RR\Hom_X^\pun(F,G)\in D(\ZZ)$ the derived complex of homomorphisms,   $\RR\HHom_X^\pun(F,G)\in D(X)$ the derived complex of sheaves of homomorphisms and $F\overset\LL\otimes G$ the derived tensor product.

Let $\pi\colon X\to \{\rm pt\}$ be the projection to a point and $\pi^{-1}\colon C(\ZZ)=C({\rm pt})\to C(X)$ the inverse image. For any $E\in C(\ZZ)$, we shall still denote by $E$ the complex $\pi^{-1}E$ on $X$; it is the constant complex: $E_x=E$ for any $x\in X$. Then, we shall denote
\[ E\otimes_\ZZ F:= (\pi^{-1}E)\otimes_\ZZ   F\] for any $E\in C(\ZZ), F\in C(X)$.

For any complex of sheaves $F$ on $X$ one has the standard resolutions (let us denote $n=\dim X$)
\[\aligned  0\to F\to C^0F\to \dots\to C^n F\to 0 
\\ 0\to C_n F\to\dots\to C_0F\to F\to 0\endaligned\]
where
\[\aligned  C^i F &=\underset{x_0<\dots <x_i}\bigoplus F_{x_i}\otimes_\ZZ \ZZ_{C_{x_0}}
\\ C_i F &=\underset{x_0<\dots <x_i}\bigoplus F_{x_0}\otimes_\ZZ \ZZ_{U_{x_i}}
\endaligned\]

These resolutions will be used to reduced the proof of some formulas, for a complex $F$, to the case where $F=E\otimes_\ZZ \ZZ_{U_x}$ or $F=E\otimes_\ZZ \ZZ_{C_x}$.

\begin{cosa}{\rm  {\em Homology and cohomology}.  For any $F\in D(X)$, we shall denote by $\RR\Gamma(X,F)$ the right derived functor  of sections of $F$; analogously, we shall denote by $\LL(X,F)$ the left derived functor of cosections of $F$ (see \cite{Sanchoetal}), and
\[ H^i(X,F)=H^i[\RR\Gamma(X,F)],\qquad H_i(X,F)=H_i[\LL(X,F)]\] the cohomology and homology groups of $F$. These may be computed with the standard resolutions, i.e.,
\[ \RR\Gamma(X,F)\simeq \Gamma(X,C^\pun F),\quad  \LL(X,F)\simeq \L(X,C_\pun F).\]
Homology and cohomology of the constant sheaf $\ZZ$ are mutually dual: $$\RR\Gamma(X,\ZZ)= \LL(X,\ZZ)^\vee 
\quad ,\quad\LL(X,\ZZ)= \RR\Gamma(X,\ZZ)^\vee.$$ 

We shall denote   $\LL_\text{\rm red}(X,\ZZ):=\operatorname{Cone}(\phi)[-1]$, where $\phi$ is the natural morphism of complexes $\phi\colon \LL(X,\ZZ)\to \ZZ$. Then
\[ \widetilde H_i(X,\ZZ):=H_i[\LL_\text{\rm red}(X,\ZZ)]\] are the reduced homology groups of $X$. Notice that, if $X$ is empty, then $\widetilde H_{-1}(X,\ZZ)=\ZZ$ and $\widetilde H_i(X,\ZZ)=0$ for any $i\neq -1$. Analogously, one defines reduced cohomology $\RR\Gamma_\text{\rm red}(X,\ZZ) $ which is dual of reduced homology. 

For any closed subset $Y$ of $X$, we shall denote $\RR\Gamma_Y(X,F)$ the right derived functor of sections with support in $Y$ and $\RR\underline\Gamma_YF$ its sheafified version. From  the exact triangle of local cohomology one obtains:
\begin{lem}\label{localcoh} For any $x\in X$ one has:
\[ \aligned\RR\Gamma_x(U_x,\ZZ)&=\RR\Gamma_{\text{\rm red}}(U_x^*,\ZZ)[-1] \\  \RR\Gamma_x(U_x,\ZZ)^\vee&=\LL_{\text{\rm red}}(U_x^*,\ZZ)[1].\endaligned \]
\end{lem}}
\end{cosa}

\begin{cosa}{\rm {\em Some formulas in derived category}.  For any continuous map $f\colon X\to Y$, the derived direct image $\RR f_*\colon D(X)\to D(Y)$ maps $D_c(X)$ into $D_c(Y)$ and $D^b(X)$ into $D^b(Y)$ and it is right adjoint of the inverse image $f^{-1}\colon D(Y)\to D(X)$.

For any open subset $j\colon U\hookrightarrow X$, $j_!\colon D(U)\to D(X)$ denotes the extension by zero functor, and it is left adjoint of $j^{-1}=$ restriction to $U$. For any $F\in D(X)$, $F_U=j_!j^{-1}F=F\otimes_\ZZ \ZZ_U$.
For any closed subset $i\colon Y\hookrightarrow X$, $F_Y=i_*i^{-1}F = F\otimes_\ZZ \ZZ_Y$, and one has an exact triangle
\[ F_{X-Y}\to F\to F_Y.\]
We shall widely use without further mention the following formulas (some of them are valid on arbitrary topological spaces (\cite{Sp}), some others are specific for finite spaces):

\begin{enumerate} 

\item For any open subset $U$ of $X$ \[\label{formula0} \RR\Gamma(U,\RR\HHom_X^\pun(F, G )) = \RR\Hom_U^\pun(F_{\vert U}, G_{\vert U} ).\]
\item For any $F,G,H\in D(X)$

\[\label{formula1} \aligned \RR\Hom_X^\pun(F\overset\LL\otimes G,H)&=\RR\Hom_X^\pun(F,\RR\HHom_X^\pun( G,H))
\\ \RR\HHom_X^\pun(F\overset\LL\otimes G,H)&=\RR\HHom_X^\pun(F,\RR\HHom_X^\pun( G,H))
\endaligned \]

\item For any open subset $U\overset j\hookrightarrow X$ 
 \[\label{formula2} 
 \RR\HHom_X^\pun(\ZZ_U,G) = \RR j_* j^{-1}G.
 \]
 \item For any $E\in D(\ZZ)$, $F\in D(X)$,
 \[\label{formula3}\RR\Hom_X^\pun(E,F)=\RR\Hom_\ZZ^\pun(E,\RR\Gamma(X,F)).\]
\item For any closed subset $i\colon Y\hookrightarrow X$,
\[\label{formula4} \aligned \RR\Hom_X^\pun(\ZZ_Y,G) &= \RR\Gamma_Y(X,G)
\\ \RR\HHom_X^\pun(\ZZ_Y,G) &= \RR{\underline\Gamma}_Y G
\\ \RR\Hom_X^\pun(F_Y,G)&=\RR\Hom_X^\pun(F,\RR\underline\Gamma_YG)
\\ i_*\RR\HHom_Y^\pun(i^{-1}F,G)&=\RR\HHom_X^\pun( F,i_*G). 
\endaligned\]
\item For any $x\in X$:
\[\label{formula5}\aligned F_x=\Gamma(U_x,F)&=\RR\Gamma(U_x,F)=\RR\Hom_X^\pun(\ZZ_{U_x},F)
\\ \RR\Hom_X^\pun(F,\ZZ_{C_x}) &= \RR\Hom_\ZZ^\pun (F_x,\ZZ).\endaligned\]
\end{enumerate}
}
\end{cosa}

We end this section by proving some technical results that will be used in the paper.

\begin{prop}\label{prop0} Let $F,G\in D(X)$ and $x$ a closed point of $X$. If $F$ and $G$ are supported at $x$, then
\[  \RR\Hom^\pun_X(F,G) =\RR\Hom^\pun_\ZZ(\RR\Gamma(X,F),\RR\Gamma(X,G)) =\RR\Hom^\pun_\ZZ(F_x, G_x). \]
\end{prop}

\begin{proof} Let $i\colon \{x\}\hookrightarrow X$ be the inclusion. Since $G$ is supported at $x$, $G=i_*i^{-1}G$ and then
\[\RR\Hom^\pun_X(F,G) = \RR\Hom^\pun_X(F,i_*i^{-1}G)= \RR\Hom^\pun_\ZZ(F_x, G_x).\] To conclude, it is enough to see that $\RR\Gamma(X,F)\to \RR\Gamma(U_x,F)=F_x$ is an isomorphism if $F$ is supported at $x$. This follows from the Mayer-Vietoris exact triangle
\[ \RR\Gamma(X,F)\to \RR\Gamma(U_x,F)\oplus \RR\Gamma(X\negmedspace -\negmedspace x,F)\to \RR\Gamma(U_x\negmedspace -\negmedspace x,F)\] because $F_{\vert X\negmedspace -  x}=0$ and $F_{\vert U_x\negmedspace -  x}=0$.
\end{proof}
\begin{prop}\label{prop1} For any $G\in D(X)$ and any $E\in D(\ZZ)$ there is a natural morphism
\[ E^\vee\otimes^\LL_\ZZ G \to\RR\HHom^\pun_X(E, G)\] which is an isomorphism if either $G$ or $E$ is bounded and coherent.
\end{prop}

\begin{proof} The natural morphism $E\otimes_\ZZ^\LL E^\vee\to\ZZ$ induces a morphism  $E\otimes_\ZZ^\LL E^\vee\otimes_\ZZ^\LL G\to G$, and then the desired morphism (adjunction between tensor product and $\RR\HHom$). The isomorphism is proved stalkwise. For any $x\in X$, one has
\[\RR\HHom^\pun_X(E, G)_x = \RR\Hom^\pun_{U_x}(E, G_{\vert U_x})=\RR\Hom^\pun_\ZZ(E,G_x)\overset\sim\longleftarrow E^\vee\otimes_\ZZ^\LL G_x\] where the last isomorphism holds if either $G_x\in \Perf(\ZZ)$ or $E\in \Perf(\ZZ)$.
\end{proof}

\begin{prop} Let $G\in D^b_c(X)$. The functor $$ \RR\HHom_X^\pun(\quad,G)\colon D(X)\to D(X)$$ maps $D_c(X)$ into $D_c(X)$ and $D^b (X)$ into $D^b (X)$.
\end{prop}

\begin{proof} Let $F\in D_c(X)$ (resp. $F\in D^b (X)$). From the exact sequence of complexes 
$$0\to C_nF\to\cdots\to C_0F\to F\to 0$$
it suffices to prove the statement for $C_iF$. This reduces the problem to the case that   $F=E\otimes_\ZZ \ZZ_{U_x}$ for some $E\in D_c(\ZZ)$ (resp. $E\in D^b (\ZZ)$) and some $x\in X$. If we denote $j\colon U_x\colon \hookrightarrow X$   the inclusion, we have:
\[\aligned \RR\HHom_X^\pun(E\otimes_\ZZ \ZZ_{U_x}, G)=\RR\HHom_X^\pun(E,\RR\HHom_X^\pun(\ZZ_{U_x}, G))& =   \RR\HHom_X^\pun(E,\RR j_* G_{\vert U_x}) \\ &\overset{\ref{prop1}}= E^\vee\otimes_\ZZ \RR j_* G_{\vert U_x}\endaligned\]
and one concludes.
\end{proof}

\section{Intervals, codimension functions and the sheaves $\ZZ_{\{x\}}$.}\label{section1}

In this section we recall some basic notions as intervals, codimension functions and the sheaves $\ZZ_{\{x\}}$ and the relations between them. In this paper, the sheaves $\ZZ_{\{x\}}$ will play the role  of   residual fields of points  on schemes. The main and elementary result of this section is to show how the morphisms  between these sheaves are related to the cohomology of intervals.

\begin{defn}{\rm For any $x <y$, we shall denote
\[\aligned  (x,y)&:=U_x^*\cap  C_y^*=\{ p\in X: x<p<y\} 
\\   [x,y]&:= U_x  \cap  C_y =\{ p\in X: x\leq p\leq y\}.\endaligned \] We say that $X$ is {\em catenary} if every interval $[x,y]$ is pure, that is, every maximal chain from $x$ to $y$ has length equal to $\dim [x,y]$. We say that an interval $(x,y)$ is a {\em homological sphere} if 
\[ H^i_{\text{\rm red}}((x,y),\ZZ)=\left\{ \aligned  0, &\text{ for }i\neq \dim (x,y) \\  \simeq \ZZ, &\text{ for } i=\dim (x,y).\endaligned\right.\]}
\end{defn}

\begin{defn} {\rm We say that $x$ {\em precedes} $y$, denoted by $x\prec y$, if $x<y$ and $(x,y)$ is empty.
A {\em codimension function} on $X$ is a map $d\colon X\to\ZZ$ such that $$d(x)=d(y)+1$$ for any $x\prec y$. }
\end{defn}

\begin{rem} {\rm If $X$ admits a codimension function, then $X$ is catenary. The converse is not true, but it holds if $X$ is either irreducible or local.}
\end{rem}

\begin{defn}{\rm For each $x\in X$, $\ZZ_{\{x\}}$ denotes the constant sheaf $\ZZ$ supported at $x$. Then
\[\ZZ_{\{x\}}=\ZZ_{U_x}\otimes_\ZZ \ZZ_{C_x}.\]}
\end{defn}

The following Proposition will be widely used in the paper.

\begin{prop}[Cohomological properties of $\ZZ_{\{x\}}$]\label{Z_p}$\,$\medskip

\begin{enumerate}
\item For any $F\in D(X)$ and any $x\in X$ one has:
\[ \RR\Hom_X^\pun(\ZZ_{\{x\}}, F)=\RR\Gamma_x(U_x,F)\quad,\quad  \RR\HHom_X^\pun(\ZZ_{\{x\}}, F)=\RR\Gamma_x(U_x,F)\otimes_\ZZ\ZZ_{C_x} \] and then
\[\Ext^i_X(\ZZ_{\{x\}}, F)=H^i_x(U_x,F) \quad,\quad   \HExt_X^i(\ZZ_{\{x\}}, F)=H^i_x(U_x,F)\otimes_\ZZ\ZZ_{C_x}.\]
\item {\rm (Escision)} For any open subset $U$ containing $x$, one has $$\RR\Hom_X^\pun(\ZZ_{\{x\}}, F)=\RR\Hom_U^\pun(\ZZ_{\{x\}}, F_{\vert U}). $$ 

\item For any $x<y$ one has:
\[ \RR\Hom_X^\pun(\ZZ_{\{x\}},\ZZ_{\{y\}})  = \RR\Gamma_{\text{\rm red}}((x,y),\ZZ)[-2] \] and then
\[ \Ext_X^i(\ZZ_{\{x\}},\ZZ_{\{y\}})  = H_{\text{\rm red}}^{i-2}((x,y),\ZZ).\]
\end{enumerate}
\end{prop}

\begin{proof} (1) Since $\ZZ_{\{x\}}=\ZZ_{U_x}\otimes \ZZ_{C_x}$,
\[\aligned \RR\Hom_X^\pun(\ZZ_{\{x\}}, F)&=\RR\Hom_X^\pun(\ZZ_{U_x}\otimes \ZZ_{C_x}, F) =  \RR\Hom_X^\pun(\ZZ_{U_x} ,\RR\underline\Gamma_{C_x} F)\\ &= \RR\Gamma(U_x, \RR\underline\Gamma_{C_x} F)= \RR\Gamma_x(U_x,F).\endaligned\]

Analogously, $$  \RR\HHom_X^\pun(\ZZ_{\{x\}},F) =  \RR\HHom_X^\pun(\ZZ_{U_x} ,\RR\underline\Gamma_{C_x}F) =\RR j_*j^{-1}\RR\underline\Gamma_{C_x}F  = \RR j_*  \RR\underline\Gamma_{x}F_{\vert U_x}. $$ Since $  \RR\underline\Gamma_{x}F_{\vert U_x}$ is supported at $x$ and its stalk at $x$ is $ \RR\Gamma_x(U_x,F)$, it follows that $ \RR\underline\Gamma_{x}F_{\vert U_x}= \RR\Gamma_x(U_x,F)\otimes_\ZZ \ZZ_{\{x\}}$ and, by projection formula, $\RR j_*  \RR\underline\Gamma_{x}F_{\vert U_x}=\RR\Gamma_x(U_x,F)\otimes_\ZZ \RR j_* \ZZ_{\{x\}}$. Then, $$\RR\HHom_X^\pun(\ZZ_{\{x\}},F)= \RR\Gamma_x(U_x,F)\otimes_\ZZ \RR j_* \ZZ_{\{x\}}  = \RR\Gamma_x(U_x,F)\otimes_\ZZ\ZZ_{C_x}.$$ \medskip

(2) follows form (1). Let us prove (3). By (1),  one has to prove that $\RR\Gamma_x(U_x,\ZZ_{\{y\}}) =\RR\Gamma_{\text{\rm red}}((x,y),\ZZ)[-2]$. From the exact sequence
\[ 0\to\ZZ_{\{y\}}\to\ZZ_{C_y}\to \ZZ_{C_y^*}\to 0\] we obtain the exact triangle
\[ \RR\Gamma_x(U_x,\ZZ_{\{y\}}) \to \RR\Gamma_x(U_x,\ZZ_{C_y})\to \RR\Gamma_x(U_x,\ZZ_{C_y^*}).\] We conclude because
\[\aligned &\RR\Gamma_x(U_x,\ZZ_{C_y}) = \RR\Gamma_x(C_y,\ZZ )\overset{\ref{localcoh}}= \RR\Gamma_{\text{\rm red}} (C_y\negthinspace-\negthinspace x,\ZZ)[-1]=0\, (\text{\rm  since } C_y\negthinspace-\negthinspace x \text{ \rm  is contractible})
\\ &\RR\Gamma_x(U_x,\ZZ_{C_y^*})  = \RR\Gamma_x(C_y^*,\ZZ ) \overset{\ref{localcoh}}=  \RR\Gamma_{\text{\rm red}} ((x,y),\ZZ)[-1].
\endaligned\]
\end{proof}

The sheaves $\{\ZZ_{U_x}\}_{x\in X}$ form a finite system of generators of $D(X)$. Let us see that also the sheaves $\{\ZZ_{C_x}\}_{x\in X}$ or $\{\ZZ_{\{x\}}\}_{x\in X}$ form a system of generators.

\begin{prop}\label{generators} Let $F\in D(X)$. The following conditions are equivalent:
\begin{enumerate}
\item $F=0$.
\item $\RR\Hom_X^\pun(\ZZ_{\{x\}},F)=0$ for any $x\in X$.
\item $\RR\Hom_X^\pun(\ZZ_{C_x},F)=0$ for any $x\in X$.
\end{enumerate}
\end{prop}

\begin{proof} (1) $\Rightarrow $ (2) is clear. Let us prove that (2) $\Rightarrow $ (3). Let us see that $\RR\Hom_X^\pun(\ZZ_{C_x},F)=0$ by induction on $\dim C_x$. If $\dim C_x=0$, then  $  C_x =  \{x\} $. Assume now that $\dim C_x>0$. From the exact sequence
\[ 0\to \ZZ_{\{x\}}\to \ZZ_{C_x}\to \ZZ_{C_x^*}\to 0\] we are reduced to prove that $\RR\Hom_X^\pun(\ZZ_{C_x^*},F)=0$. By induction, $\RR\Hom_X^\pun(\ZZ_{C_y},F)=0$ for any $y<x$, and then $\RR\Hom_X^\pun(E\otimes_\ZZ \ZZ_{C_y},F)=0$ for any $E\in D(\ZZ)$ and any $y<x$. Hence, $\RR\Hom_X^\pun(C^i\ZZ_{C_x^*},F)=0$ for any $i$, and we conclude by the exact sequence of complexes
\[ 0\to \ZZ_{C_x^*} \to C^0\ZZ_{C_x^*}\to \cdots \to C^n\ZZ_{C_x^*}\to 0.\]

(3) $\Rightarrow $ (1). Let us prove that $\RR\Hom_X^\pun(G,F)=0$ for any $G\in D(X)$. If $G=\ZZ_{C_x}$, this holds by hypothesis. Then, it holds for $G=E\otimes_\ZZ \ZZ_{C_x}$ for any $E\in D(\ZZ)$ and any $x\in X$. Then the result holds for $C^iG$ for any $i$ and any $G$. Again, the exact sequence $0\to G\to C^0G\to\cdots\to C^nG\to 0$ allows us to conclude.
\end{proof}

\section{Local and global dualizing complexes}\label{Local-global-section}

In this section we introduce local and global dualizing complexes. Our main interest is the local dualizing complex of a local space.

\subsection{Global dualizing complex} 

Let us recall here the notion of global dualizing complex (see \cite{Na}, \cite{Sanchoetal}, \cite{Curry}).

\begin{defn}{\rm Let $f\colon X\to S$ be a continuous map between finite spaces. The functor
\[ \aligned D(X)&\longrightarrow D(S)\\ F&\longmapsto \RR f_*F \endaligned\] has a right adjoint
\[f^!\colon D(S)\to D(X)\] and we denote $D_{X/S}:=f^!\ZZ$ and name it {\em relative  dualizing complex} of $X$ over $S$. If $S$ is just a point, then $D_{X/S}$ is denoted by $D_X$ and named {\em global dualizing complex} of $X$. By definition, one has
\[\qquad\qquad \qquad\RR\Hom^\pun_X(F, D_X)=\RR\Hom_\ZZ^\pun(\RR\Gamma(X,F),\ZZ)\qquad (\text{\rm Global duality isomorphism})\] for any $F\in D(X)$.}
\end{defn}

We recall the construction of $f^!$, since we shall use it later. Let $G$ be a sheaf on $Y$. The functor
\[ \aligned\Shv(X)&\longrightarrow \{\text{\rm Abelian groups}\} \\ F&\longmapsto \Hom_Y(f_* C^nF,G)\endaligned\] is right exact and takes filtered direct limits into filtered inverse limits. Hence it is representable by a sheaf $f^{-n}G$ on $X$. A morphism $G\to G'$ induces a morphism $f^{-n}G\to f^{-n}G'$. The natural morphism $f_* C^{n-1}F \to f_*C^nF$ induces a morphism $f^{-n}G\to f^{-n+1}G$. Thus, if $G$ is a complex of sheaves on Y, then we have a bicomplex $f^{-p}G^q$, whose associated simple complex is denoted by $f^\nabla G$. One has an isomorphism
\[ \Hom_Y^\pun(f_*C^\pun F,G) = \Hom^\pun_X(F, f^\nabla G)\]
which extends to the same isomorphism for a complex of sheaves $F$ on $X$. Then, $f^!G$ is defined by
\[  f^!G:= f^\nabla I(G)\] with $G\to I(G)$ an injective resolution.

For any continous map $f\colon X\to S$ one has a natural isomorphism $D_X=f^!D_S$. By adjunction, one has a morphism
\[ \RR f_*D_X\to D_S.\]

\subsection{Local dualizing complex}


Let $f\colon X\to S$ be a continuous map and $Y$ a closed subset of $X$. 

\begin{defn}{\rm For any sheaf $F$ on $X$, we define the sheaf $f^Y_*F$ on $S$ by:
\[ (f^Y_*F)(V)=\Gamma_{Y\cap f^{-1}(V)}(f^{-1}(V), F) \] for any open subset $V$ of $S$. In other words
\[ f_*^YF=f_*\HHom_X(\ZZ_Y,F).\] We obtain then a (left-exact) functor
\[ f_*^Y\colon \Shv(X)\longrightarrow\Shv(S) \] and its right derived functor $$\RR f^Y_*\colon D(X)\to D(S) $$ which can be described as $\RR f_*^Y (F)=\RR f_*\RR\HHom_X^\pun(\ZZ_Y,F)$.}
\end{defn}

\begin{thm} \label{loc-rel-dualizing}  The functor $\RR f_*^Y$ has a right adjoint 
\[ f_{(X,Y)}^!\colon D(S)\to D(X).\] 
Moreover, if we denote $U=X-Y$, $j\colon U\hookrightarrow X$ the inclusion and $f_U=f\circ j$, then for any $G\in\D(S)$ one has an exact triangle 
\[ \RR j_* f_U^! G\to f^! G\to f_{(X,Y)}^!G.\]  
\end{thm}

\begin{proof} Let $G$ be a sheaf on $Y$. The functor
\[ \aligned\Shv(X)&\longrightarrow \{\text{\rm Abelian groups}\} \\ F&\longmapsto \Hom_Y(f^Y_* C^nF,G) \endaligned\] is right exact and takes filtered direct limits into filtered inverse limits. Hence it is representable by a sheaf $f_{(X,Y)}^{-n}G$ on $X$.

A morphism $G\to G'$ induces a morphism $f_{(X,Y)}^{-n}G\to f_{(X,Y)}^{-n}G'$. The natural morphism $f^Y_* C^{n-1}F \to f_*^YC^nF$ induces a morphism $f_{(X,Y)}^{-n}G\to f_{(X,Y)}^{-n+1}G$. Thus, if $G$ is a complex of sheaves on Y, then we have a bicomplex $f_{(X,Y)}^{-p}G^q$, whose associated simple complex is denoted by $f^\nabla_{(X,Y)}G$. One has an isomorphism
\begin{equation}\label{dual-const} \Hom_Y^\pun(f^Y_*C^\pun F,G) = \Hom^\pun_X(F, f^\nabla_{(X,Y)}G)\end{equation}
which extends to the same isomorphism for a complex of sheaves $F$ on $X$. Then, it suffices to define
\[  f_{(X,Y)}^!G:= f_{(X,Y)}^\nabla I(G)\] with $G\to I(G)$ an injective resolution. We leave the reader to check the details. Now, from the exact sequence of complexes
\[ 0\to f^Y_*C^\pun F \to f_* C^\pun F \to f_* j_*C^\pun F_{\vert U}\to 0\] one obtains  the exact sequence of complexes 
\[ 0\to j_* f_U^\nabla I(G) \to f^\nabla I(G)\to f_{(X,Y)}^\nabla I(G)\to 0\] hence the exact triangle
\[ \RR j_* f_U^! G\to f^! G\to f_{(X,Y)}^!G.\] 
\end{proof}

\begin{defn}{\rm We shall denote $D_{X/S}^Y:= f_{(X,Y)}^!\ZZ$. It is called the {\em local relative dualizing complex of $X$ over $S$ along $Y$}. If $S$ is a point, it will be denoted by $D_X^Y$ and named {\em local dualizing complex of $X$ along $Y$}.
}
\end{defn}

\begin{thm}\label{closed-loc-rel-dualizing} Let $f\colon X\to S$ be a continuous map, $Y  $   a   closed subset of $X$. For any closed subset $i\colon K\hookrightarrow X$, one has:
\[ i_* D_{K/S}^{Y\cap K} =\RR\HHom_X^\pun(\ZZ_K, D_{X/S}^Y).\] 
More generally, for any $G\in D(S)$, one has
\[ \RR\HHom_X^\pun(\ZZ_K, f_{(X,Y)}^!G) = i_* (f_{(K,K\cap Y)}^!G).\]
\end{thm}

\begin{proof} For any open subset $V$ of $S$ and any $F\in D(X)$, one has
\[ \Gamma_{Y\cap f^{-1}(V)}(f^{-1}(V) , F_K)= \Gamma_{ K\cap Y\cap f^{-1}(V)}(K\cap f^{-1}(V), F_{\vert K})
\] and then $f^Y_* F_K =f^{Y\cap K}_* F_{\vert K} $. The result  follows from adjunction.
\end{proof}
 For the rest of the paper we shall just consider the following particular case:

\begin{defn} {\rm Let $X$ be a   finite local space, $\0\in X$ the closed point  and $X^*=X\negmedspace -\negmedspace \{\0\}$. The complex  $D_X^\0$ will be called the {\em local dualizing complex of $X$}. By definition
\begin{equation}\label{localduality}  \RR\Hom_X^\pun(F,D_X^\0) =\RR\Hom_\ZZ^\pun(\RR\Gamma_\0(X,F),\ZZ)  
\end{equation} which will be refered to as the {\em  local duality isomorphism}.
 One has an exact triangle
\begin{equation}\label{local-global} \RR j_* D_{X^*}\to \ZZ_{\{\0\}} \to D_X^\0\end{equation} where $j\colon X^* \hookrightarrow X$ is the inclusion, because $D_X=\ZZ_{\{{\0}\}}$ (see \cite{Sanchoetal}). Moreover, for any closed subset $i\colon K\hookrightarrow X$ one has:
\begin{equation}\label{local-closed} i_* D_K^\0=\RR\HHom_X^\pun(\ZZ_K, D_X^\0) \end{equation} i.e., $D_K^\0= i^{-1} \RR\HHom_X^\pun(\ZZ_K, D_X^\0)$.}
\end{defn}

The following results show the relation between the local dualizing complex and the sheaves $\ZZ_{\{x\}}$.

\begin{rem}\label{remlocdual}{\rm By local duality, $$\RR\Hom^\pun_X(\ZZ_{\{{\0}\}}, D_X^\0)=\ZZ$$ and then
\[ \RR\HHom^\pun_X(\ZZ_{\{{\0}\}}, D_X^\0)= \ZZ_{\{{\0}\}}\] by Proposition \ref{Z_p}.}
\end{rem}

\begin{prop}\label{loc-coh-dualizante} For any $x>\0  $ one has:
\[\RR\Hom_X^\pun(\ZZ_{\{x\}},D_X^\0) = \RR\Hom_X^\pun(\ZZ_{\{{\0}\}},\ZZ_{\{x\}})^\vee = \LL_{\text{\rm red}}((\0,x),\ZZ)[2].\] Hence,
\[ \Ext^{-i}_X(\ZZ_{\{x\}},D_X^\0)=\widetilde H_{i-2}((\0,x),\ZZ).\]
\end{prop}
\begin{proof} The first isomorphism is local duality; the second follows from Proposition \ref{Z_p}.
\end{proof}

%


\begin{prop}\label{restricciondualizantelocal} Let $ x<y$ and $j\colon U_y\hookrightarrow U_x$ the inclusion. One has:
\[ \RR\Hom_{U_x}^\pun(j^{-1} D_{U_x}^x, D_{U_y}^y)=\RR\Hom_X^\pun(\ZZ_{\{x\}},\ZZ_{\{y\}}).\]
\end{prop}

\begin{proof} One has
\[\aligned  \RR\Hom_{U_y}^\pun(j^{-1} D_{U_x}^x, D_{U_y}^y)& \overset{\ref{localduality}}=\RR\Hom_{U_y}^\pun(\ZZ_{\{y\}}, j^{-1} D_{U_x}^x)^\vee = \RR\Hom_{U_x}^\pun (j_!\ZZ_{\{y\}},   D_{U_x}^x)^\vee =\\ & = \RR\Hom_{U_x}^\pun ( \ZZ_{\{y\}},   D_{U_x}^x)^\vee    \overset{\ref{localduality}}   = \RR\Hom_{U_x}^\pun(\ZZ_{\{x\}},\ZZ_{\{y\}})^{\vee\vee} = \\ & = \RR\Hom_{U_x}^\pun(\ZZ_{\{x\}},\ZZ_{\{y\}})  \overset{\ref{Z_p}}{=\negthinspace=} \RR\Hom_X^\pun(\ZZ_{\{x\}},\ZZ_{\{y\}})\endaligned\]

\end{proof}

%
%
%
%


\section{Canonical complexes}\label{canonicalcomplex-section}

A. Grothendieck introduced the notion of   ``dualizing complex'' on schemes, as a complex that yields reflexivity. In this paper we introduce the analogous notion for finite spaces, but to avoid confusion with the dualizing complexes  of the previous section, we shall call them canonical complexes. Its main properties mimic that of Grothendieck's on schemes.

\begin{defn}{\rm  A {\em canonical complex} on $X$ is a complex $\Omega\in D^b_c(X)$ such that, for any $F\in D_c(X)$, the natural morphism  
 \[F\to \RR\HHom_X^\pun(\RR\HHom_X^\pun(F,\Omega),\Omega)\] is an isomorphism. This isomorphism will be refered to as the {\em reflexivity isomorphism}. A finite space is called {\em dualizable} if a canonical complex on $X$ exists (as we shall see later, if it exists, it is essentially unique). 
}\end{defn}

\begin{defn} {\rm For any $\Omega\in D(X)$, let us denote $D:=\RR\HHom_X^\pun(\quad,\Omega)\colon D(X)\to D(X)$. We say that $F\in D(X)$ is {\em $\Omega$-reflexive}  if $F\to D(D(F))$ is an isomorphism.}
\end{defn}

\begin{rem}\label{dualidad}{\rm Let $\Omega$ be a canonical complex. For any $F,G\in D_c(X)$ one has an isomorphism
\[ \RR\HHom_X^\pun(F,G) =\RR\HHom_X^\pun(D(G),D(F))\] and then an isomorphism
\[ \RR\Hom_X^\pun(F,G) =\RR\Hom_X^\pun(D(G),D(F)).\] 

Indeed, $\RR\HHom_X^\pun(D(G),D(F))= \RR\HHom_X^\pun(D(G)\overset\LL\otimes F,\Omega) = \RR\HHom_X^\pun(F,D(D(G)))$, and one concludes because $G=D(D(G))$.}
\end{rem}

\begin{prop}\label{reflexive-generators} Let $\Omega\in D^b_c(X)$. The following conditions are equivalent.
\begin{enumerate} \item $\Omega$ is a canonical complex.
\item Any $F\in D^b_c(X)$ is $\Omega$-reflexive.
\item For any $x\in X$, $\ZZ_{U_x}$ is $\Omega$-reflexive.
\item  For any $x\in X$, $\ZZ_{C_x}$ is $\Omega$-reflexive.
\item For any $x\in X$, $\ZZ_{\{x\}}$ is $\Omega$-reflexive.
\end{enumerate}
\end{prop}

\begin{proof} In order to prove that $F$ is reflexive, it suffices to prove that either $C^iF$ is reflexive for any $i$, or $C_iF$ is reflexive for any $i$. Also notice that for any $E\in D(\ZZ)$ and any $F\in D^b_c(X)$ one has
\begin{equation}\label{D-linear} D(E\overset\LL\otimes F)=E^\vee\overset\LL\otimes D(F).\end{equation}
Indeed, $D(E\overset\LL\otimes F)=\RR\HHom_X^\pun(E\overset\LL\otimes F,\Omega)= \RR\HHom_X^\pun(E ,D(F)) \overset{\ref{prop1}}= E^\vee\overset\LL\otimes D(F)$.

(3) $\Rightarrow $ (1). If $\ZZ_{U_x}$ is reflexive, then $E\otimes_\ZZ \ZZ_{U_x}$ is reflexive for any $E\in D_c(\ZZ)$, by \eqref{D-linear}. Then $C_iF$ is reflexive for any $F\in D_c(X)$ and any $i$, hence $F$ is reflexive.

(4) $\Rightarrow $ (1). This is analogous, replacing $\ZZ_{U_x}$ by $\ZZ_{C_x}$ and $C_iF$ by $C^iF$ in the previous proof.

(5) $\Rightarrow $ (4). By induction on $\dim C_x$. If it is zero, then $C_x=\{x\}$. If $\dim C_x>0$, from the exact sequence 
\[ 0\to \ZZ_{\{x\}} \to\ZZ_{C_x}\to \ZZ_{C_x^*} \to 0\] we are reduced to  prove that $\ZZ_{C_x^*}$ is reflexive. By induction, $\ZZ_{C_y}$ is reflexive for any $y<x$; by \eqref{D-linear}, $E\otimes_\ZZ \ZZ_{C_y}$ is reflexive for any $E\in D_c(\ZZ)$ and then $C^i\ZZ_{C_x^*}$ is reflexive for any $i$ and we conclude.

The remaining implications are trivial.
\end{proof}

\begin{thm}\label{dualizante-openclosed}  Let $\Omega\in D^b_c(X)$ be a canonical complex.
\begin{enumerate} \item If $U$ is an open subset of $X$, then $\Omega_{\vert U}$ is a canonical complex on $U$.
\item If $K\overset i\hookrightarrow X$ is a closed subset, then $  i^{-1}\RR\HHom_X^\pun(\ZZ_K,\Omega)$ is a canonical complex on $K$.
\end{enumerate}
\end{thm} 

\begin{proof} (1) It is immediate from the equality $\RR\HHom_X^\pun(F,\Omega)_{\vert U}= \RR\HHom_U^\pun(F_{\vert U},\Omega_{\vert U})$ and the fact that any $G\in D_c(U)$ is the restriction to $U$ of an $F\in D_c(X)$ (for example, $F=j_!G$, with $j$ the inclusion of $U$ in $X$).

(2) Let us denote $\Omega^K =i^{-1}\RR\HHom_X^\pun(\ZZ_K,\Omega)$, so $i_*\Omega^K = \RR\HHom_X^\pun(\ZZ_K,\Omega)$. Let us first see that \begin{equation}\label{Omega^K} i_*\RR\HHom_K^\pun(F,\Omega^K)=\RR\HHom_X^\pun(i_*F,\Omega)\end{equation} for any $F\in D(K)$. Indeed, since $F=i^{-1}i_*F$, one has
\[\aligned  i_*\RR\HHom_K^\pun(F,\Omega^K)& =  \RR\HHom_X^\pun(i_*F,i_*\Omega^K) = \RR\HHom_X^\pun(i_*F,  \RR\HHom_X^\pun(\ZZ_K,\Omega)) 
\\ &= \RR\HHom_X^\pun(i_*F\otimes\ZZ_K,\Omega) =  \RR\HHom_X^\pun(i_*F ,\Omega).\endaligned\]
Now, let $F\in D_c(K)$ and let us see that $F\to\RR\HHom_K^\pun(\RR\HHom_K^\pun(F,\Omega^K),\Omega^K)$ is an isomorphism. It suffices to see it after applying $i_*$. Now,
\[ \aligned i_*\RR\HHom_K^\pun(\RR\HHom_K^\pun(F,\Omega^K),\Omega^K) &\overset{\eqref{Omega^K}} =  \RR\HHom_X^\pun(i_*\RR\HHom_K^\pun(F,\Omega^K),\Omega ) 
\\ &\overset{\eqref{Omega^K}} =  \RR\HHom_X^\pun( \RR\HHom_X^\pun(i_*F,\Omega ),\Omega )=i_*F.\endaligned\]
\end{proof}

\begin{cor} Any locally closed subspace of a dualizable space is dualizable. Any simplicial complex is dualizable. Any projective simplicial complex is dualizable. More generally, any locally closed subset of a simplicial (or projective simplical) complex is dualizable.
\end{cor}

\begin{proof} The first statement follows from Theorem \ref{dualizante-openclosed}. It is proved in \cite{ST} that $\A^n_{\FF_1}$ is dualizable (in fact, it is proved that $\ZZ_{\{g\}}$ is a canonical complex, with $g$ the generic point of $\A^n_{\FF_1}$). Hence, any locally closed subset of $\A^n_{\FF_1}$ is dualizable.
\end{proof}

\begin{rem}[Compatibility with Stanley--Reisner correspondence]\label{St-Reis-compatib} {\em Let $K\subseteq \A^n_{\FF_1}$ be a simplicial complex and $S_K\subseteq \A^n_\ZZ$ the associated closed subscheme, according to Stanley--Reisner correspondence. It is proved in \cite{ST} that one has  continuous maps $\pi \colon \A^n_\ZZ\to \A^n_{\FF_1}$,  $\pi_K\colon S_K\to K$ and   exact functors 
\[ \pi^\stella\colon\Shv(\A^n_{\FF_1})\longrightarrow\Qcoh(\A^n_\ZZ)\quad ,\quad  \pi_K^\stella\colon\Shv(K)\longrightarrow\Qcoh(S_K). \] Moreover, one has  an isomorphism
\[ \pi^\stella \RR\HHom_{\A^n_{\FF_1}}(\ZZ_K, \ZZ_{\{g\}}) = \RR\HHom_{\A^n_{\ZZ}}(\OO_{S_K}, \omega_{\A^n_{\ZZ}})\] where $\omega_{\A^n_\ZZ}$ is the canonical module of the scheme $\A^n_\ZZ$, i.e., the highest differentials: $\omega_{\A^n_\ZZ}=\Omega^n_{\A^n_\ZZ}$. Consequently, if $\Omega_K$ is a canonical complex on $K$, then $\pi^\stella(\Omega_K)$ is a dualizing complex of the scheme $S_K$.}
\end{rem}

\begin{thm}[Uniqueness of  canonical complexes] \label{unicity}  Assume that $X$ is connected. Let $\Omega\in D^b_c(X)$ be a canonical complex. A complex $\Omega'\in D^b_c(X)$ is a canonical complex if and only if $$\Omega'=\Omega\otimes_\ZZ\Lc [r]$$ for some $r\in \ZZ$ and some invertible sheaf $\Lc$ (i.e., a sheaf locally isomorphic to the constant sheaf $\ZZ$).
\end{thm}

\begin{proof} The converse is trivial. Let us prove the direct. For any $F\in D_c(X)$, let us denote 
\[ D(F)=\RR\HHom_X^\pun(F,\Omega)\quad,\quad D'(F)=\RR\HHom_X^\pun(F,\Omega').\]
Let us define $L=D'D(\ZZ)\in D^b_c(X)$, and let us first prove that
\begin{equation}\label{lema-unicidad} F\overset\LL\otimes L = D'D (F)
\end{equation} for any $F\in D_c(X)$. Indeed, since $D'$ satisfies reflexivity, it suffices to see that $D'(F\overset\LL\otimes L) =  D(F)$. Now,
\[ \aligned D'(F\overset\LL\otimes L) & = \RR\HHom_X^\pun( F\overset\LL\otimes L,\Omega') = \RR\HHom_X^\pun( F, D' ( L)) = \RR\HHom_X^\pun( F, D (\ZZ)) 
\\ &= \RR\HHom_X^\pun( F, \Omega  )=D(F)\endaligned\] as wanted. Taking $F=L':=DD'(\ZZ)$ in equality \eqref{lema-unicidad}, we obtain that $L'\overset\LL\otimes L=\ZZ$. Taking the stalk at a point $x\in X$, we obtain $L'_x\overset\LL\otimes L_x=\ZZ$ and then $L_x=\ZZ[r_x]$ for some integer $r_x$ (Lemma \ref{lem0}). For any $y>x$, the commutativity of the diagram
\[
\xymatrix{ L'_x\overset\LL\otimes L_x \ar[r]^{\quad\sim}\ar[d] & \ZZ \ar[d]^{\id} \\ L'_y\overset\LL\otimes L_y \ar[r]^{\quad\sim} & \ZZ}\] yields that $r_x=r_y$. Hence $r_x$ is locally constant on $x$. Since $X$ is connected, it is constant. In conclusion, $L=\Lc[r]$ for some invertible sheaf $\Lc$ and some integer $r$. Finally, taking $F=\Omega=D(\ZZ)$ in equality \eqref{lema-unicidad}, we obtain
\[ \Omega \otimes\Lc[r] = D'DD(\ZZ)=D'(\ZZ)=\Omega'\] and the theorem is proved.
\end{proof}

Let us see now that a canonical complex has an associated codimension function (the proof that it is in fact a codimension function will be done later).

\begin{prop}\label{cod-fun} Let $\Omega$ be a canonical complex on $X$. For each $x\in X$, there is an integer $\phi_x$ such that:
\[ \aligned \RR\Hom_X^\pun(\ZZ_{\{x\}},\Omega)&=\ZZ[-\phi_x]
\\  \RR\HHom_X^\pun(\ZZ_{\{x\}},\Omega)&=\ZZ_{C_x}[-\phi_x]   
\\  \RR\HHom_X^\pun(\ZZ_{C_x} ,\Omega)&= \ZZ_{\{x\}}[-\phi_x].\endaligned\] 
\end{prop} 

\begin{proof}  By reflexivity, one has
\[ \ZZ_{C_x}\overset\sim\longrightarrow \RR\HHom_X^\pun( \RR\HHom_X^\pun(\ZZ_{C_x},\Omega),\Omega ) \simeq \RR\HHom_X^\pun (\RR\underline\Gamma_{C_x}\Omega, \RR\underline\Gamma_{C_x}\Omega)\] and taking the stalk at $x$
\[ \ZZ \simeq \RR\Hom^\pun_{U_x}(\RR\underline\Gamma_{x}\Omega_{\vert U_x} ,\RR\underline\Gamma_{x}\Omega_{\vert U_x} )\overset{\ref{prop0}}{=\negthinspace=} \RR\Hom^\pun_\ZZ(\RR\Gamma_x(U_x,\Omega),\RR\Gamma_x(U_x,\Omega)).\] By Lemma \ref{lem0}, $\RR\Gamma_x(U_x,\Omega)=\ZZ[-\phi_x]$ for some integer $\phi_x$. The second formula follows from Proposition \ref{Z_p}, and the third follows from the second by reflexivity.
\end{proof}

\begin{defn} {\rm The map $$\aligned \phi\colon X&\longrightarrow \ZZ\\  x&\mapsto \phi_x\endaligned$$  is called the {\em codimension function} associated to $\Omega$.}

\end{defn}

\begin{thm}[Biduality]\label{thm1} Let $\Omega$ be a canonical complex on $X$, with codimension function $\phi$. For each $x\in X$, $F\in D_c(X)$, one has:
\[ \RR\Hom_X^\pun(\ZZ_{\{x\}},F)=\RR\Hom^\pun_{U_x}(F_{\vert U_x},\Omega_{\vert U_x})^\vee [-\phi_x]\]  and then
\[ \RR\Hom^\pun_{U_x}(F_{\vert U_x},\Omega_{\vert U_x}) =\RR\Hom_X^\pun(\ZZ_{\{x\}},F)^\vee [-\phi_x].\]
Consequently, the stalkwise description of $\Omega$ is given by the formula:
\[ \Omega_x = \RR\Hom_X^\pun(\ZZ_{\{x\}},\ZZ)^\vee[-\phi_x] \overset{\ref{localcoh}}= \LL_{\text{\rm red}}(U_x^*,\ZZ)[1-\phi_x].\]

\end{thm}

\begin{proof} One has  
\[ \aligned \RR\Hom_X^\pun(\ZZ_{\{x\}},F) &\overset{\ref{dualidad}}=  \RR\Hom_X^\pun(D(F),D(\ZZ_{\{x\}})) \overset{\ref{cod-fun}}= \RR\Hom_X^\pun(D(F),\ZZ_{C_x})[-\phi_x]
\\ &= \RR\Hom_\ZZ^\pun(D(F)_x,\ZZ )[-\phi_x].\endaligned\]
Since $D(F)_x=\RR\Hom^\pun_{U_x}(F_{\vert U_x},\Omega_{\vert U_x})$, we conclude the first equality.  Taking dual $\,^\vee$ we obtain the second one. The stalkwise description of $\Omega$ follows by taking $F=\ZZ$ in the second equality. 
 

\end{proof}

\begin{cor}[Compatibility of canonical complexes and local dualizing complexes] If $\Omega$ is a canonical complex of $X$, with codimension function $\phi$, then for any $x\in X$:
\[ \Omega_{\vert U_x} = D_{U_x}^x[ -\phi_x].\]
\end{cor}

This corollary says that on a local space there is only one candidate (up to shift) to be a canonical complex: the local dualizing complex. On an irreducible space, there is only one candidate: the constant sheaf $\ZZ$ supported at the generic point, as we shall prove now.

\begin{prop}\label{dualizanteirreducible} Let $X$ be an irreducible finite space and $g$ the generic point. If $\Omega$ is a canonical complex on $X$, with codimension function $\phi$, then
\[\Omega = \ZZ_{\{g\}} [-\phi_g].\]  
\end{prop}

\begin{proof} By Theorem \ref{thm1}, \begin{equation}\label{Omega_g} \Omega_g=\ZZ [-\phi_g].\end{equation}   Let us consider the exact sequence
\[ 0\to \Omega_{\{g\}}\to\Omega\to \Omega_{X-g}\to 0.\] Applying $\RR\HHom_X^\pun(\quad,\Omega)$ we obtain the exact triangle
\[ \RR\HHom_X^\pun(\Omega_{X-g},\Omega) \to   \RR\HHom_X^\pun(\Omega,\Omega)\to   \RR\HHom_X^\pun(\Omega_{\{g\}},\Omega).\] We have that $\RR\HHom_X^\pun(\Omega,\Omega)=\ZZ$; let us see that also $\RR\HHom_X^\pun(\Omega_{\{g\}},\Omega)=\ZZ$. Indeed, if $j\colon \{g\}\hookrightarrow X$ is the inclusion, then
\[ \RR\HHom_X^\pun(\Omega_{\{g\}},\Omega)=   j_* \RR\Hom_\ZZ^\pun(\Omega_g,\Omega_g)\overset{\eqref{Omega_g}}  =   j_*\ZZ=\ZZ.\] From the exact triangle, we conclude that $\RR\HHom_X^\pun(\Omega_{X-g},\Omega)=0$, and then $\Omega_{X-g}=0$. Hence $\Omega=\Omega_{\{g\}}=\ZZ_{\{ g\}}[-\phi_g]$.
\end{proof}

Our aim now is to show some necessary topological conditions that a finite space $X$ must satisfy in order to admit a canonical complex. These conditions will also be  sufficient when the space is either local or irreducible.

\begin{prop}\label{intervalo} Let $\Omega$ be a canonical complex on $X$ with codimension function  $\phi$. For any $x<y$ one has:
\[ \RR\Gamma_{\text{\rm red}}((x,y),\ZZ)[-2]=\RR\Hom^\pun_X(\ZZ_{\{x\}},\ZZ_{\{y\}}) = \ZZ[\phi_y-\phi_x].\]
\end{prop}

\begin{proof} The first equality is given in Proposition \ref{Z_p}. For the second, 
\[ \aligned \RR\Hom^\pun_X(\ZZ_{\{x\}},\ZZ_{\{y\}}) \overset{\ref{dualidad}} = \RR\Hom^\pun_X(D(\ZZ_{\{y\}}),D(\ZZ_{\{x\}})) &\overset{\ref{cod-fun}}= \RR\Hom_X^\pun(\ZZ_{C_y} , \ZZ_{C_x} )[\phi_y-\phi_x]
\\ & =\ZZ[\phi_y-\phi_x].\endaligned\]
\end{proof}

\begin{thm}\label{top-conditions} Let $\Omega$ be a canonical complex on $X$ with codimension function  $\phi$. Then:
\begin{enumerate} \item $\phi$ is a codimension function: $ \phi_x= \phi_y+1$ for any $x\prec y$. In particular, $X$ is catenary.
\item For any $x<y$, the inverval $(x,y)$ is a  homological sphere.
\end{enumerate}
\end{thm}

\begin{proof} (1) If $x\prec y$, then $(x,y)$ is empty, so $\RR\Gamma_{\text{\rm red}}((x,y),\ZZ)=\ZZ[1]$. By Proposition \ref{intervalo}, $\phi_y-\phi_x=-1$.

(2)  By Proposition \ref{intervalo}, $\RR\Gamma_{\text{\rm red}}((x,y),\ZZ)=\ZZ[2-\phi_x+\phi_y]$. It remains to prove that $\dim (x,y)=\phi_x-\phi_y-2$. Since $\phi$ is a codimension function, $\dim [x,y]=\phi_x-\phi_y$. We conclude because $\dim (x,y)=\dim [x,y]-2.$ 
\end{proof}

We shall  see now that the converse of this theorem holds when $X$ is either local or irreducible.

\begin{thm}\label{existenciadualizante1} Assume that $X$ is  local with closed point $\0$. Assume further that

{\rm (a)} $X$ is catenary.
 
{\rm (b)} For any $x<y$, the interval $(x,y)$ is a homological  sphere.  \medskip

Then:
\begin{enumerate}\item The local dualizing complex $D_X^\0$ is a canonical complex, with codimension function $\phi_x=-\dim C_x$.
\item For any $x\in X$, one has
\[ [D_X^\0]_{\vert U_x}\simeq D_{U_x}^x [\dim C_x].\]
\end{enumerate}   
\end{thm}

\begin{proof} Let us  prove (2). One has 
\[\aligned \RR\Hom_{U_x}^\pun((D_{X}^\0)_{\vert {U_x}} , D_{U_x}^x[\dim C_x])&\overset{\ref{restricciondualizantelocal}}= \RR\Hom_X(\ZZ_\0,\ZZ_{\{x\}})[\dim C_x]
\\ & \overset{\ref{Z_p}}= \RR\Gamma_{\text{\rm red}}((\0,x),\ZZ)[\dim C_x -2]
\\ & \overset{(b)}\simeq \ZZ.\endaligned\]
Hence, the element $1\in \ZZ$ gives a morphism  $\psi\colon (D_{X}^\0)_{\vert {U_x}} \to D_{U_x}^x[\dim C_x]$. To prove that $\psi$ is an isomorphism, it suffices (by Proposition \ref{generators}) to see it after applying $\RR\Hom_{U_x}^\pun(\ZZ_{\{q\}},\quad)$ for any $q\in U_x$. Now,
\[\aligned \RR\Hom_{U_x}^\pun(\ZZ_{\{q\}}, (D_{X}^\0)_{\vert {U_x}} ) &\overset{\ref{loc-coh-dualizante}}= \LL_{\text{\rm red}}((\0,q),\ZZ)[2]\overset{(b)}\simeq \ZZ[\dim [\0,q]]
\\ \RR\Hom_{U_x}^\pun(\ZZ_{\{q\}}, D_{U_x}^x )&\overset{\ref{loc-coh-dualizante}}= \LL_{\text{\rm red}}((x,q),\ZZ)[2] \overset{(b)}\simeq \ZZ[\dim [x,q]]\endaligned\] and we conclude because $\dim[x,q]+\dim C_x=\dim[\0,q]$ by catenarity.

%

Now, let us prove (1). Let $F\in D_c(X)$ and let us see that
\[F\to\RR\HHom_X^\pun(\RR\HHom_X^\pun(F,D_X^\0),D_X^\0)\] is an isomorphism. We proceed stalkwise. The fiber at the closed point $\0$ of the right term  is
\[\aligned \RR\Hom_X^\pun(\RR\HHom_X^\pun(F,D_X^\0),D_X^\0) &\overset{\ref{localduality}}=  \RR\Hom_X^\pun(\ZZ_{\{\0\}}, \RR\HHom_X^\pun(F,D_X^\0))^\vee 
\\ &=  \RR\Hom_X^\pun(F, \RR\HHom_X^\pun(\ZZ_{\{\0\}},D_X^\0))^\vee  
\\ & \overset{\ref{remlocdual}}= \RR\Hom_X^\pun(F, \ZZ_{\{\0\}}))^\vee = (F_\0)^{\vee\vee}=F_\0 \endaligned\] as wanted. The fiber at a non closed point $x$ is an isomorphism because the restriction of $D_X^\0$ to $U_x$ is isomorphic to $D_{U_x}^x[\dim C_x]$ (by (2)), which is a canonical complex on $U_x$ by induction on the dimension. Finally, the codimension function of $D_X^\0$ is $\phi_x=-\dim C_x$, because  $\phi_{\0}=0$, since $\RR\Hom_X^\pun(\ZZ_{\{\0\}},D_X^\0)=\ZZ$ by local duality.
\end{proof}

\begin{cor} A finite space is locally dualizable if and only if it is catenary and every interval is a homological sphere.
\end{cor}

\begin{thm}\label{existenciadualizante2} Assume that $X$ is  irreducible with generic point $g$. Assume further that

{\rm (a)} $X$ is catenary.
 
{\rm (b)} For any $x<y$, the interval $(x,y)$ is a homological sphere.  \medskip

Then $\ZZ_{\{g\}}$ is a canonical complex, with codimension function $\phi_x= \dim U_x$.  
\end{thm}

\begin{proof}  For any $x\in X$, $U_x$ is local, hence it has a canonical complex, by Theorem \ref{existenciadualizante1}; since $U_x$ is irreducible, $\ZZ_{\{g\}}$ is a canonical complex in $U_x$ (Proposition \ref{dualizanteirreducible}). Hence, $\ZZ_{\{g\}}$ is a canonical complex on $X$ because its restriction to each $U_x$ is a canonical complex. The codimension function is $\phi_x=\dim U_x$, because $\phi_g=0$, since $\RR\Hom_X^\pun(\ZZ_{\{g\}},\ZZ_{\{g\}})=\ZZ$.
\end{proof}

\begin{cor} Let $X$ be a catenary finite space such that $(x,y)$ is a homological sphere for any $x<y$. Then, for any point $x\in X$, $D_{U_x}^x$ is a canonical complex of $U_x$ and $\ZZ_{\{x\}}$ is a canonical complex of $C_x$.
\end{cor}

\begin{thm}[Local and global duality on a local dualizable space]\label{loc-duality-thm} Let $X$ be a local and dualizable finite space with closed point $\0$. Let $D_X^\0$ be the local dualizing complex and for each $F\in D_c(X)$, let us denote
\[ D(F)=\RR\HHom_X^\pun(F,D_X^\0).\]  

{\rm\bf (A)} For each $x\in X$ one has (let us denote $c_x=\dim C_x$):
\begin{enumerate}\item[(a)] $  D(F)_x = \RR\Hom_X^\pun(\ZZ_{\{x\}},F)^\vee[c_x] \overset{\ref{Z_p}}= \RR\Gamma_x(U_x,F)^\vee[c_x].$   \medskip
\item[(b)] $F_x  = \RR\Hom_X^\pun(\ZZ_{\{x\}},D(F))^\vee[c_x] \overset{\ref{Z_p}} = \RR\Gamma_x(U_x,D(F))^\vee[c_x].$ 
\end{enumerate} In particular,
\begin{enumerate}
\item[(a')] $  [D_X^\0]_x =\RR\Hom_X^\pun(\ZZ_{\{x\}},\ZZ)^\vee[c_x]  \overset{\ref{Z_p}} = \RR\Gamma_x(U_x,\ZZ)^\vee [c_x]\overset{\ref{localcoh}} =\LL_{\text{\rm red}}(U_x^*,\ZZ)[1+c_x].$ \medskip
\item[(b')] $ \RR\Gamma_x(U_x,D_X^\0) \overset{\ref{Z_p}} =\RR\Hom_X^\pun(\ZZ_{\{x\}},D_X^\0)=\ZZ[c_x].$
\end{enumerate}\medskip
 
{\rm\bf (B)} Let us denote $X^*=X\negmedspace -\negmedspace\0$. Then $D_{X^*}:=(D_X^\0)_{\vert X^*}[-1]$ is the global dualizing complex of $X^*$, hence
 \[ \RR\Hom^\pun_{X^*}(G,D_{X^*})=\RR\Gamma(X^*,G)^\vee\] for any $G\in D(X^*)$.
\end{thm}

\begin{proof}  (a) follows from the second formula of Theorem \ref{thm1}, because $D_X^\0$ is a canonical complex with codimension function $\phi_x=-\dim C_x$.   Replacing $F$ by $D(F)$, we obtain   (b). The   formulae (a')-(b') are obtained by taking $F=\ZZ$. (B) follows from \ref{local-global}.
%
\end{proof}

The following corollary is a consequence of Theorems \ref{loc-duality-thm}, \ref{existenciadualizante1} and \ref{dualizante-openclosed}, and it  shows that abstract simplicial complexes (which we have called projective simplicial complexes) behave like projective schemes.

\begin{cor} Let $K^*$ be an $n$-dimensional projective simplicial complex. The global dualizing complex $D_{K^*}$ is a canonical complex and,  locally, it agrees with the local dualizing complex, i.e.: for each $x\in X$ one has
\[ {D_{K^*}}_{\vert U_x}=D_{U_x}^x[\dim C_x].\]
In the case $K^*=\PP^{n}_{\FF_1}$, one has $D_{\PP^n_{\FF_1}}=\ZZ_{\{g\}}[n]$, where $g$ is the generic point of $\PP^n_{\FF_1}$. In general, if $i\colon K^*\hookrightarrow \PP^r_{\FF_1}$ is a closed subset, then
\[ D_{K^*}=i^{-1}\RR\HHom^\pun_{\PP^r_{\FF_1}}(\ZZ_{K^*},\ZZ_{\{g\}})[r].\]
\end{cor}

\subsubsection{Products}

Let $X,Y$ be two finite spaces. Let us consider the direct product $X\times Y$ and the natural projections $\pi_X\colon X\times Y\to X$, $\pi_Y\colon X\times Y\to Y$. For any $F\in D(X)$, $G\in D(Y)$, let us denote
\[ F\boxtimes G=(\pi_X^{-1}F)\overset\LL\otimes (\pi_Y^{-1}G).\]  
For any open subsets $U$ of $X$ and $V$ of $Y$, one has $(F\boxtimes G)_{\vert U\times V}= F_{\vert U}\boxtimes G_{\vert V}$. For any point $(x,y)\in X\times Y$ one has
\[ \ZZ_{\{(x,y)\}}=  \ZZ_{\{x\}}\boxtimes  \ZZ_{\{y\}}.\]
\begin{lem} One has \[ \RR\Gamma(X\times Y, F\boxtimes Y)=\RR\Gamma(X,F)\overset\LL\otimes  \RR\Gamma(Y,G).\]
\end{lem}

\begin{proof} Let us denote $p_X\colon X\to \{*\}$, $p_Y\colon Y\to \{*\}$, the projections to a point. Then, using base change and projection formula (twice) (see \cite{Sanchoetal})
\[ \aligned \RR\Gamma(X\times Y, F\boxtimes Y)&=\RR{p_X}_*\RR{\pi_X}_* \left[(\pi_X^{-1}F)\overset\LL\otimes (\pi_Y^{-1}G)\right]   = 
\RR{p_X}_*( F \overset\LL\otimes \RR{\pi_X}_*    \pi_Y^{-1}G)
\\ & = \RR{p_X}_*( F \overset\LL\otimes p_X^{-1}\RR{p_Y}_*     G) = \RR{p_X}_*  F  \overset\LL\otimes \RR{p_Y}_*     G\endaligned\]
\end{proof}

For any $F_1,F_2\in D(X)$ and $G_1,G_2\in D(Y)$ one has a natural morphism
\begin{equation}\label{prod-morphism}
\RR\HHom_X^\pun(F_1,F_2)\boxtimes \RR\HHom_Y^\pun(G_1,G_2)\longrightarrow \RR\HHom_{X\times Y}^\pun(F_1\boxtimes  G_1,F_2\boxtimes G_2)
\end{equation} and then a morphism
\begin{equation}\label{prod-morphism2}
\RR\Hom_X^\pun(F_1,F_2)\overset\LL\otimes \RR\Hom_Y^\pun(G_1,G_2)\longrightarrow \RR\Hom_{X\times Y}^\pun(F_1\boxtimes  G_1,F_2\boxtimes G_2)
\end{equation}

\begin{prop}\label{prod-iso} Assume that $F_1,F_2\in D^b_c(X)$ and $G_1,G_2\in D^b_c(Y)$. Then \eqref{prod-morphism} and \eqref{prod-morphism2}  are isomorphisms.
\end{prop}

\begin{proof} To see that \eqref{prod-morphism} is an isomorphism, we proceed stalkwise. For any $(x,y)\in X\times Y$, one has $U_{(x,y)}=U_x\times U_y$. Taking the stalk at $(x,y)$ in the morphism \eqref{prod-morphism} one obtains the morphism \eqref{prod-morphism2} for the spaces $U_x,U_y$ and the sheaves $F_i,G_j$ restricted to them. In conclusion, it suffices to prove that \eqref{prod-morphism2} is an isomorphism. 

(a) First assume that 
\[ F_1 =E_1\otimes_\ZZ \ZZ_{U_{x_1}}\quad , \quad F_2=E_2\otimes_\ZZ \ZZ_{U_{x_2}}\] \[ 
 G_1=V_1\otimes_\ZZ \ZZ_{U_{y_1}}\quad , \quad G_2=V_2\otimes_\ZZ \ZZ_{U_{y_2}}\]
with $E_i, V_i\in D^b_c(\ZZ)$, $x_i\in X$,  $y_i\in Y$. In this case the isomorphism \eqref{prod-morphism2} is given by direct computation.

(b) By (a), one has that \eqref{prod-morphism2} is an isomorphism for the complexes $C_{i_1} F_1$, $C_{i_2} F_2$, $C_{j_1} G_1$, $C_{j_2} F_2$, for any integers $i_1,i_2,j_1,j_2$.

From the exact sequence of complexes $0\to C_n F_1\to \cdots\to C_0F_1\to F_1\to 0$ (and the analogous ones for $F_2,G_1,G_2$), we conclude that \eqref{prod-morphism} is an isomorphism for any $F_1,F_2\in D^b_c(X)$, $G_1,G_2\in D^b_c(Y)$.
\end{proof}

\begin{prop}\label{prod-dualizable} $X\times Y$ is dualizable if and only if $X$ and $Y$ are dualizable.
\end{prop}

\begin{proof} Assume that $X$ and $Y$ are dualizable, with canonical complexes $\Omega_X$, $\Omega_Y$, and let us see that $\Omega_X\boxtimes\Omega_Y$ is a canonical complex on $X\times Y$. By Proposition \ref{prod-iso}
\[ D(F\boxtimes G)=D(F)\boxtimes D(G)\] for any $F\in D_c(X)$, $G\in D_c(Y)$, and then $F\boxtimes G$ is reflexive. In particular, the sheaves $\ZZ_{\{(x,y)\}}$ are reflexive, because $\ZZ_{\{(x,y)\}}=\ZZ_{\{x\}}\boxtimes\ZZ_{\{y\}}$. We conclude then by Proposition \ref{reflexive-generators}.

Conversely, assume that $X\times Y$ is dualizable. Let $y\in Y$ be a closed point. Then $X=X\times\{ y\}$ is a closed subset of $X\times Y$, hence it is dualizable. Analogously, $Y$ is also dualizable.

\end{proof}

The same Proposition holds for  local dualizing complexes:

\begin{prop}\label{local-prod} Let $X$ and $Y$ be local spaces with closed points $x_0, y_0$. Then
\[ D_{X\times Y}^{(x_0,y_0)}=D_X^{x_0}\boxtimes D_Y^{y_0}   .\]
\end{prop}

\begin{proof} Let us first define a morphism $\psi\colon D_X^{x_0}\boxtimes D_Y^{y_0}\to D_{X\times Y}^{(x_0,y_0)}$. One has equalities
\[\aligned \RR\Hom_{X\times Y}^\pun(D_X^{x_0}\boxtimes D_Y^{y_0}, D_{X\times Y}^{(x_0,y_0)}) &\overset{\ref{localduality}}= \RR\Hom_{X\times Y}^\pun( \ZZ_{\{(x_0,y_0)\}},   D_X^{x_0}\boxtimes D_Y^{y_0})^\vee 
\\ &\overset{\ref{prod-iso}} = 
\RR\Hom_{X }^\pun( \ZZ_{\{x_0\}},   D_X^{x_0})^\vee \overset\LL\otimes \RR\Hom_{Y }^\pun( \ZZ_{\{y_0\}},   D_Y^{y_0})^\vee
\\ &   \overset{\ref{localduality}}= \RR\Hom^\pun_X(D_X^{x_0}, D_X^{x_0}) \overset\LL\otimes  \RR\Hom^\pun_Y(D_Y^{y_0}, D_Y^{y_0}).\endaligned.\]

Hence, the identity morphisms $D_X^{x_0}\to D_X^{x_0}$, $D_Y^{y_0}\to D_Y^{y_0}$, give the desired morphism $\psi$. In order to see that $\psi$ is an isomorphism, it suffices, by Proposition \ref{generators}, to see it after applying $\RR\Hom_{X\times Y}^\pun(\ZZ_{\{(x ,y )\}},\quad)$ for any $(x,y)\in X\times Y$. Now,
\[\aligned \RR\Hom_{X\times Y}^\pun(\ZZ_{\{(x ,y )\}},D_X^{x_0}\boxtimes D_Y^{y_0}) &\overset{\ref{prod-iso}} = \RR\Hom_{X}^\pun(\ZZ_{\{x\}},D_X^{x_0}) \overset\LL\otimes \RR\Hom_{ Y}^\pun(\ZZ_{\{ y  \}},  D_Y^{y_0}) 
\\ &\overset{\ref{localduality}}= \RR\Hom_{X}^\pun(\ZZ_{\{x_0\}}, \ZZ_{\{x\}} )^\vee \overset\LL\otimes \RR\Hom_{ Y}^\pun(\ZZ_{\{y_0\}}, \ZZ_{\{ y  \}})^\vee\endaligned
\] and
\[\aligned \RR\Hom_{X\times Y}^\pun(\ZZ_{\{(x ,y )\}},D_{X\times Y}^{(x_0,y_0)}) &\overset{\ref{localduality}}= 
 \RR\Hom_{X\times Y}^\pun(\ZZ_{\{(x_0 ,y_0 )\}} ,\ZZ_{\{(x ,y )\}})^\vee 
 \\ &\overset{\ref{prod-iso}} = \RR\Hom_{X}^\pun(\ZZ_{\{x_0\}}, \ZZ_{\{x\}} )^\vee \overset\LL\otimes \RR\Hom_{ Y}^\pun(\ZZ_{\{y_0\}}, \ZZ_{\{ y  \}})^\vee
 \endaligned\]

\end{proof}

\section{Cohen--Macaulayness}\label{CM-section}

Following Grothendieck, we develop a Cohen--Macaulayness theory on finite spaces from that of canonical complexes.

\subsection{Support, dimension and depth}
\begin{defn}{\rm  For any $F\in D(X)$, we recall that its {\em support} is defined by:
\[ \supp(F)=\{ x\in X: F_x\neq 0\}\] and its {\em dual support} by:
\[ \dsupp(F)=\{ x\in X:\RR\Hom_X^\pun(\ZZ_{\{x\}},F)\neq 0\}.\] The {\em dimension} of $F$ is defined by:
\[\dim(F)=\dim (\overline{\supp(F)}).\] Finally, we say that $F$ is {\em   pure} if $\overline{\supp(F)}$ is pure.
}
\end{defn}

\begin{rem} {\rm Notice that
\[ \overline{\supp(F)}=\{ x\in X: F_{\vert U_x}\neq 0\}.\]}
\end{rem}

\begin{rem}{\rm  $H^i(X,F)=0$ for any $i>\dim F$, because $H^i(X,F)=H^i(S,F_{\vert S})$, with $S=\overline{\supp(F)}$. }
\end{rem}

\begin{prop} One has:
\begin{enumerate}\item $\overline{\supp(F)} =\overline{\dsupp(F)}$.
\item If $X$ has a canonical complex $\Omega$, then, for any $F\in D_c(X)$,
\[ \dsupp(F)=\supp(D(F))\]  where $D(F)=\RR\HHom_X^\pun(F,\Omega)$.
\end{enumerate}
\end{prop}

\begin{proof} (1). If $x\in\dsupp(F)$ then $\RR\Hom_X^\pun(\ZZ_{\{x\}},F)\neq 0$, hence $F_{\vert U_x}\neq 0$ so $x\in \overline{\supp(F)}$. To conclude, it suffices to see that any generic point of the support of $F$ belongs to the dual support of $F$. Let $g$ be such a generic point. Then $F_{\vert U_g}$ is supported at the closed point $g$, hence
\[\RR\Hom^\pun_X(\ZZ_{\{g\}},F)=\RR\Gamma_g(U_g, F)=F_g\neq 0\] and then $g\in\dsupp(F)$.

(2). It follows from the equalities 
\[\aligned \RR\Hom_X^\pun(\ZZ_{\{x\}},F)&\overset{\ref{dualidad}}=\RR\Hom_X^\pun (D(F),D(\ZZ_{\{x\}}))\overset{\ref{cod-fun}} = \RR\Hom_X^\pun (D(F),\ZZ_{C_x})[-\phi_x] 
\\ & = \RR\Hom_\ZZ^\pun(D(F)_x,\ZZ)[-\phi_x].\endaligned \]
\end{proof}

\begin{defn}{\rm Let $F$ be a sheaf on $X$. For each $x\in \dsupp(F)$ we define the {\em depth of $F$ at $x$} by:
\[ \depth_x (F) =  \text{\rm smallest } k  \text{ \rm such that } \Ext^k_X(\ZZ_{\{x\}}, F)\neq 0\]
or equivalently, the smallest $k$ such that $H^k_x(U_x,F)\neq 0$. Then, the {\em depth } of $F$ is defined by:
\[ \depth (F) = \underset{x\in \dsupp(F)}{\text{\rm min }} \{ \depth_x(F) +\dim (C_x)\}\] 
}
\end{defn}

\begin{prop} One has: \begin{enumerate} \item $\depth(F)\leq \dim (F)$.
\item $\depth_x(F)\leq \dim (F_{\vert U_x})$ for any $x\in\dsupp(F)$..
\end{enumerate}
\end{prop}

\begin{proof} (1) Let $g$ be a generic point of $\supp (F)$ such that $\dim C_g=\dim (F)$. Then $\depth_g(F)=0$, so $\depth_g(F)+\dim C_g =\dim (F)$ and the result follows.

(2) $H^i_x(U_x,F)=$ for any $i>\dim (F_{\vert U_x})$, hence the result.
\end{proof}

\begin{rem} {\rm Assume that $F$ is pure. Then, the following conditions are equivalent:
\begin{enumerate} \item $\depth (F)=\dim (F)$.
\item For every $x\in\dsupp(F)$, $\depth_x(F)=\dim (F_{\vert U_x})$.
\end{enumerate}  Indeed, the purity of $F$ implies that $\dim (F_{\vert U_x})+\dim C_x=\dim F$ for any $x\in\overline{\supp(F)}$, and one concludes easily.}
\end{rem}

\subsection{Cohen-Macaulayness on local and dualizable spaces}

Let $X$ be a local space, with closed point $\0$. Let us further assume that $X$ is dualizable. Then, the local dualizing complex $D_X^\0$ is a canonical complex, and we shall denote 
\[ D(F)=\RR\HHom_X^\pun(F,D_X^\0)\] for any $F\in D(X).$

\begin{defn}{\rm    We say that a sheaf $F$ on $X$ is {\em torsion free} if $F_x$ is torsion free for any $x\in X$. We say that $F$ is a {\em torsion sheaf} if $F_x$ is a torsion abelian  group for any $x\in X$.}
\end{defn}

\begin{rem}{\rm Any finitely generated abelian group $G$ decomposes as a direct sum $$G=L\oplus T,$$ where $L$ is a   free $\ZZ$-module of finite rank and $T$ is a torsion group. If either $L=0$ or $T=0$, then we say that $G$ is {\em unmixed}. Moreover, one has
\[ G^\vee=L^\vee\oplus T^\vee,\quad L^\vee=L^*,\quad T^\vee\simeq   T[-1]\] where $L^*=\Hom_\ZZ(L,\ZZ)$, which is a free $\ZZ$-module of the same rank than $L$.}
\end{rem}

\begin{defn}\label{CM-def} {\rm  A finitely generated sheaf $F$ is called {\em Cohen--Macaulay} if $D(F)$ is a sheaf, i.e.,
\[ D(F)=F^\#[r]\] for some sheaf $F^\#$ on $X$ and some integer $r$. In this case $F^\#$ is called the {\em dual sheaf} of $F$.}
\end{defn}

\begin{rem}{\rm Notice that $\dsupp(F^\#)=\supp(F)$ and $\overline{\supp(F)} = \overline{\supp(F^\#)}$. By reflexivity $F^\#$ is also Cohen--Macaulay, of the same dimension than $F$.}
\end{rem}

Let us see that a Cohen--Macaulay sheaf is generically unmixed.

\begin{prop}[Generic structure of a Cohen--Macaulay sheaf]\label{generic-CM} Let $F$ be a Cohen--Macaulay sheaf, $D(F)=F^\#[r]$. For any generic point $g$ of $\supp(F)$, $F_g$ is unmixed:  either $F_g$ is torsion free or $F_g$ is a torsion group. In the first case, $r=\dim C_g$; in the second case $r+1=\dim C_g$. Then
\[ \supp(F)=S_{\rm free}\cup S_{\rm tor}\] where $S_{\rm free}$ (resp. $S_{\rm tor}$) is the union of the irreducible components of $\supp(F)$ such that $F$ is torsion free (resp. a torsion group) at its generic point, so 
\[\overline{\supp(F)}=\overline{S_1}\cup \overline{S_2}\] and $\overline{S_1}$ (resp. $\overline{S_2}$) is pure of dimension $r$ (resp. of dimension $r+1$). Hence
\[\dim (F) =\left\{\aligned &r,\text{ \rm if } S_2=\emptyset
\\ &r+1, \text{ \rm if } S_2\neq \emptyset.\endaligned\right.\]
\end{prop}

\begin{proof} Let $g$ be a generic point of $\supp(F)$ and let $F_g=L\oplus T$ be the decomposition of $F_g$ in its torsion free and torsion parts.  By Theorem \ref{loc-duality-thm},
\[   \qquad  [F^\#]_g[r]  =  (F_g)^\vee [\dim C_g]\simeq L^*[\dim C_g] \oplus T [-1+\dim C_g] . 
  \] Then, either $r=\dim C_g$ (and then $T=0$) or $r=-1+\dim C_g$ (and then $L^*=0$. i.e., $L=0$).
\end{proof}

\begin{defn}{\rm We say that a sheaf $F$ is {\em generically torsion free} if $F_g$ is torsion free for any generic point $g$ of $\supp(F)$.}
\end{defn}

\begin{thm}\label{CM-freesheaf} Let $F$ be a finitely generated and generically torsion free sheaf on $X$. For each $x\in X$, let us denote $d_x=\dim(F_{\vert U_x})$. If $F$ is Cohen--Macaulay, then:
\begin{enumerate} \item $F$ is pure.
\item For any $x\in X$, $H^i_x(U_x,F)=0$ for $i\neq d_x$, and $H^{d_x}_x(U_x,F)$ is torsion free.
\item For any $x\in\dsupp(F)$, $\depth_x (F) =d_x$.
\item $\depth (F)=\dim (F)$.
\item $D(F)=F^\# [\dim(F)]$, where $F^\#$ is a torsion free and Cohen-Macaulay sheaf, and
\[  [F^\#]_x =H^{d_x}_x(U_x,F)^*\quad,\quad  H^{d_x}_x(U_x,F^\#)= (F_x)^*.
 \] 
 \item $F$ is torsion free.
\end{enumerate}

Conversely, if  $F$ satisfies {\rm (1)} and {\rm (2)}, then $F$ is Cohen--Macaulay.
\end{thm}

\begin{proof} By Proposition \ref{generic-CM}, $F$ is pure and $D(F)=F^\#[\dim F]$.  By Theorem \ref{loc-duality-thm}, for any $x\in X$:
\[   \RR\Gamma_x(U_x,F)^\vee = [F^\#]_x[\dim F-\dim C_x]. 
  \]  If $x\in\overline{\supp(F)} $, then $\dim F-\dim C_x=d_x$, because $F$ is pure. Decomposing $[F^\#]_x=L\oplus T$ in its  free and torsion parts and taking dual $\,^\vee$, we obtain
 \[ \RR\Gamma_x(U_x,F)\simeq L^*[-d_x]\oplus T[-1-d_x]\]  and then
\[\aligned & H^i(U_x,F)=0, \text{ for } i<d_x,
\\ & H^{d_x}_x(U_x,F)=L^* \text{ is torsion free},
\\ & T =H^{d_x+1}_x(U_x,F)=0,  \text{ hence }  F^\# \text{ is torsion free.}
\endaligned\] Replacing $F$ by $F^\#$ above, we obtain that $ H^{d_x}_x(U_x,F^\#)=\Hom_\ZZ(F_x,\ZZ)$ and that $F$ is torsion free, so everything in (1)-(6) is proved.

The converse follows easily from the equality $D(F)_x=\RR\Gamma_x(U_x,F)^\vee[\dim C_x]$; indeed, by (2),   $\RR\Gamma_x(U_x,F)=H^{d_x}_x(U_x,F)[-d_x]$ and $ H^{d_x}_x(U_x,F)$ is torsion free, hence $\RR\Gamma_x(U_x,F)^\vee= H^{d_x}_x(U_x,F)^*[d_x]$, so  $D(F)_x=H^{d_x}_x(U_x,F)^*[d_x+\dim C_x]\overset{(1)}= H^{d_x}_x(U_x,F)^*[\dim F]$, and then $D(F)=F^\#[\dim F]$, with $F^\#$ defined by $(F^\#)_x=H^{d_x}_x(U_x,F)^*$.

\end{proof}

Let us prove now the analogous theorem for torsion sheaves.

\begin{thm}\label{CM-torsionsheaf} Let $F$ be a finitely generated  and torsion   sheaf on $X$. For each $x\in X$, let us denote $d_x=\dim(F_{\vert U_x})$. If $F$ is Cohen--Macaulay, then:
\begin{enumerate} \item $ F $ is pure.
\item For any $x\in X$, $H^i_x(U_x,F)=0$ for $i\neq d_x$.
\item For any $x\in\dsupp(F)$, $\depth_x (F) =d_x$.
\item $\depth (F)=\dim (F)$.
\item $D(F)=F^\#[ \dim(F)-1]$, where $F^\#$ is a torsion  and Cohen-Macaulay sheaf, and
\[  [F^\#]_x \simeq H^{d_x}_x(U_x,F) \quad,\quad  H^{d_x}_x(U_x,F^\#)\simeq  F_x.
 \] 
\end{enumerate} Conversely, if $F$ is pure and satisfies either {\rm (2)} or {\rm (3)}  or {\rm (4)}, then $F$ is Cohen--Macaulay. 
\end{thm}

\begin{proof}  By Proposition \ref{generic-CM}, $F$ is pure and $D(F)=F^\#[\dim F -1]$. By   Theorem \ref{loc-duality-thm},
\[    [F^\#]_x[\dim F -1]  = \RR\Gamma_x(U_x,F)^\vee [\dim C_x] 
  \]  and then $[F^\#]_x$ is a torsion group (notice that $H^i_x(U_x,F)$ is a torsion abelian group for any $i$, because  $F$ is a torsion sheaf). Now, $\dim F - \dim C_x= d_x$ because $F$ is pure. Then, taking dual $\,^\vee$ we obtain:
\[ \RR\Gamma_x(U_x, F)=\RR\Hom^\pun_\ZZ([F^\#]_x,\ZZ)[1-d_x] \simeq  [F^\#]_x [-d_x]\] and then:
\[\aligned  H^i(U_x,F)&=0, \text{ for } i\neq d_x 
\\ H^{d_x}_x(U_x,F)&\simeq [F^\#]_x.\endaligned \] Replacing $F$ by $F^\#$ above, we obtain $ H^{d_x}_x(U_x,F^\#)=F_x$ and everything in (1)-(5) is proved.

For the converse, notice that (2) and (3) are equivalent and also (3) and (4) are equivalent if $F$ is pure.  As in the previous theorem, the result follows easily from the equality $D(F)_x=\RR\Gamma_x(U_x,F)^\vee[\dim C_x]$.

\end{proof}

\subsection{Cohen--Macaulay local spaces}\label{CM-local}

Let $X$ be a local and dualizable space, with closed point $\0$. Notice that the constant sheaf $\ZZ$ is torsion free and $\supp(\ZZ)=X$.

\begin{defn} {\rm We say that $X$ is {\em Cohen--Macaulay} if the constant sheaf $\ZZ$ is Cohen--Macaulay. This means that \[ D_{X}^\0=\omega_{X}[\dim X]\] for some sheaf $\omega_{X}$ on $X$, which is called  {\em canonical sheaf} of $X$.}
\end{defn}
\begin{rem}\label{CMlocal}{\rm (a) If $X$ is Cohen--Macaulay, then $U_x$ is Cohen--Macaulay for any $x\in X$, and
\[\omega_{U_x}\simeq {\omega_X}_{\vert U_x}\] because ${D_X^\0}_{\vert U_x}\simeq D_{U_x}^x[-\dim C_x]$.

(b) For any $x\in X$, $C_x$ is Cohen--Macaulay and $\omega_{C_x}=\ZZ_{\{x\}}$.}
\end{rem}

By Theorem \ref{CM-freesheaf}, one has

\begin{thm}\label{CM-space} Let $X$ be a local and dualizable space. If $X$ is Cohen-Macaulay, then:
\begin{enumerate}\item $X$ is pure.
\item For any $x\in X$ one has: $$H^i_x(U_x,\ZZ)=0 \text{ for  } i\neq \dim U_x \text{ and } H^{\dim U_x}_x(U_x,\ZZ) \text{ is torsion free}.$$ 
\item For any $x\in X$, one has (let us denote $d_x=\dim U_x$)
\[ \omega_{X,x}=H^{d_x}_x(U_x,\ZZ)^*\quad ,\quad H^{d_x}_x(U_x,\omega_X)=\ZZ.\] Moreover, $\omega_X$ is a torsion free and Cohen--Macaulay sheaf, hence
\[ H^i_x(U_x,\omega_X)=0\text{ for any } i\neq d_x.\]
\end{enumerate}
Conversely, if $X$ satisfies {\rm (1)} and {\rm (2)}, then it is Cohen-Macaulay.
\end{thm}

\begin{rem} {\rm Since $\RR\Gamma_x(U_x,\ZZ)^\vee \overset{\ref{localcoh}}=\LL_{\text{\rm red}}(U_x^*,\ZZ)[1]$, and $H_{\dim U_x^*}(U_x^*,\ZZ)$ is always torsion free, the statement (2) of Theorem \ref{CM-space} is equivalent to:
\begin{enumerate}\item [(2')]   $ \widetilde H_i(U_x^*,\ZZ)=0 \text{ for any } i\neq\dim U_x^*$. 
\end{enumerate}
When (2') holds for a point $x\in X$, we shall say that $X$ is {\em Cohen--Macaulay at $x$}.}
\end{rem}

Let us see that the converse part of Theorem \ref{CM-space} can be improved: condition (2) (or (2')) suffices for $X$ to be Cohen-Macaulay.

\begin{cor}\label{CM-ateverypoint} A local and dualizable space $X$ is Cohen--Macaulay if and only if 
\[  \widetilde H_i(U_x^*,\ZZ)=0 \text{ for any } x\in X \text{ and any } i\neq\dim U_x^*.\] In other words, $X$ is Cohen--Macaulay if and only if it is Cohen--Macaulay at every point.
\end{cor}

\begin{proof} We only have to prove the converse. We proceed by induction on $\dim X$. Notice that any local space of dimension $\leq 1$ is dualizable and Cohen-Macaulay, hence we may assume that $\dim X>1$. Let $D$ be the local dualizing complex of $X$.  By induction, $U_x$ is Cohen--Macaulay for any $x\neq \0$, hence $D_{\vert U_x}=F^x[r_x]$ for some sheaf $F^x$ on $U_x$ and some integer $r_x$. This integer is locally constant on $x\in X^*$. Since $X^*$ is connected (because $X$ has dimension $>1$ and it is Cohen--Macaulay at $\0$), $r_x$ is constant. In particular, $r=r_g=\dim C_g$ for any generic point $g$ of $X$. Hence $X$ is pure and we conclude by Theorem \ref{CM-space}.
\end{proof}

\begin{thm}\label{CM-closed} Let $i\colon K\hookrightarrow X$ be a closed subset. The following conditions are equivalent:
\begin{enumerate} \item $K$ is Cohen--Macaulay.
\item $\ZZ_K$ is a Cohen--Macaulay sheaf of $X$.
\end{enumerate}
Moreover, if $X$ is Cohen--Macaulay, then 
\[ K \text{ is Cohen--Macaulay }\Leftrightarrow \HExt^i_X(\ZZ_K,\omega_X)=0 \text{ for any } i\neq c:=\dim X-\dim K\] and in this case $i_*\omega_K = \HExt^c_X(\ZZ_K,\omega_X)$.
\end{thm}

\begin{proof} It follows from the isomorphism \eqref{local-closed}: $i_*D_K^\0=\RR\HHom_X^\pun(\ZZ_K, D_X^\0)$.
\end{proof}

It is proved in \cite{ST} that this theorem, when $X=\A^n_{\FF_1}$ and then $K$ is a simplicial complex, yields that $K$ is Cohen--Macaulay (in our sense) if and only if it satisfies Reisner's criterion (vanishing of homology of  links).

More generally, we can characterize Cohen--Macaulay sheaves on a Cohen--Macaulay local space:

\begin{prop} Let $X$ be a local and Cohen--Macaulay space with canonical sheaf $\omega_X$. A sheaf $F$ on $X$ is Cohen--Macaulay if and only if, there exists an integer $c$ such that
\[ \HExt_X^i(F,\omega_X)=0 \text{ for any } i\neq c.\] In this case, one has: 
$$\aligned &c=\left\{\aligned &\dim X-\dim F, \text{ if }  F \text{ is generically torsion free } \\  &\dim X-\dim F+1, \text{ otherwise},\endaligned\right.
\\  &\HExt_X^c(F,\omega_X) \text{ is  Cohen--Macaulay,} 
\\ &F=\HExt^c_X(\HExt_X^c(F,\omega_X),\omega_X)\endaligned.$$
\end{prop}

\begin{rem}{\rm Following the notations of  Remark \ref{St-Reis-compatib}, if  $K\subseteq \A^n_{\FF_1}$ is a Cohen--Macaulay simplicial complex and $S_K\subseteq \A^n_\ZZ$ is the associated Cohen--Macaulay scheme, then:
\[\pi^\stella\omega_K=\omega_{S_K} \] where $\omega_{S_K}$ denotes the canonical module of the scheme $S_K$.}
\end{rem}

\subsubsection{Cohomological properties of the canonical sheaf}$\,$\medskip

We  now reproduce our previous results in a more ``down to earth'' way. By Theorems \ref{loc-duality-thm} and \ref{CM-closed}, we have:
\begin{thm}\label{omega1} Let $X$ be an $n$-dimensional local and dualizable Cohen--Macaulay finite space with canonical sheaf $\omega_X$. For each $x\in X$,  $d_x=\dim U_x$ and $\omega_{U_x}={\omega_X}_{\vert U_x}$. Then:  \medskip

{\rm\bf (A)} For any $F\in D_c(X)$ and any $x\in X$, one has:
\[ \aligned \RR\Hom_{U_x}^\pun(F_{\vert U_x},\omega_{U_x})&=\RR\Hom_X^\pun(\ZZ_{\{x\}}, F)^\vee [-\dim U_x]
\\ \RR\Hom_X^\pun(\ZZ_{\{x\}}, F) &=   \RR\Hom_{U_x}^\pun(F_{\vert U_x},\omega_{U_x})^\vee [-\dim U_x]
\endaligned\] hence  short exact sequences  
\[\aligned 0\to \Ext^1_\ZZ(H^{d_x+1-i}_x(U_x, F),\ZZ  )  \to\Ext^i_{U_x}(F_{\vert U_x},\omega_{U_x}) \to \Hom_\ZZ(H^{d_x-i}_x(U_x, F),\ZZ)\to 0
\\ 0\to \Ext^1_\ZZ(\Ext_{U_x}^{d_x+1-i} (F_{\vert U_x},\omega_{U_x}),\ZZ  )  \to H^i_x(U_x,F) \to \Hom_\ZZ(\Ext_{U_x}^{d_x-i} (F_{\vert U_x}\omega_{U_x}),\ZZ)\to 0.\endaligned\]
Taking $F=\ZZ_K$ for some closed subset $K$ of $X$ and $x=\0$, we obtain   exact sequences
\[\aligned  0\to \Ext^1_\ZZ(H^{n+1-i}_\0(K, \ZZ),\ZZ  )  \to H^i_{K}(X,\omega_{X}) \to \Hom_\ZZ(H^{n-i}_\0(K, \ZZ),\ZZ)\to 0
\\ 0\to \Ext^1_\ZZ(H^{n+1-i}_{K}(X,\omega_{X}),\ZZ  )  \to H^i_{\0}(K,\ZZ) \to \Hom_\ZZ(H^{n-i}_{K}(X,\omega_{X}),\ZZ)\to 0
\endaligned.\]
 
{\rm\bf (B)} The sheaf $$\omega_{X^*}:={(\omega_X)}_{\vert X^*}$$ is a global dualizing sheaf on $X^*$, i.e., $\omega_{X^*}[n-1]$ is the global dualizing complex of $X^*$. Thus, for any $F\in D(X^*)$ one has (notice that $n-1=\dim X^*$)
\[ \RR\Hom_{X^*}^\pun(F,\omega_{X^*})=\RR\Hom_\ZZ^\pun(\RR\Gamma(X^*,F),\ZZ)[1-n]\] and then  short exact sequences 
\[\aligned 0\to \Ext^1_\ZZ(H^{n-i}(X^*, F),\ZZ  )  \to\Ext^i_{X^*}(F ,\omega_{X^*}) \to \Hom_\ZZ(H^{n-1-i}(X^*, F),\ZZ)\to 0
\\ 0\to \Ext^1_\ZZ(\Ext^{n-i}_{X^*}(F ,\omega_{X^*}),\ZZ  )  \to H^{i}(X^*, F) \to \Hom_\ZZ(\Ext^{n-1-i}_{X^*}(F ,\omega_{X^*}),\ZZ)\to 0
\endaligned.\] In particular, taking $F=\ZZ$, one has   exact sequences
\[\aligned 0\to \Ext^1_\ZZ(H^{n-i}(X^*, \ZZ),\ZZ  )  \to H^i (X^* ,\omega_{X^*}) \to \Hom_\ZZ(H^{n-1-i}(X^*, \ZZ),\ZZ)\to 0
\\ 0\to \Ext^1_\ZZ(H^{n-i}(X^* ,\omega_{X^*}),\ZZ  )  \to H^i (X^* ,\ZZ) \to \Hom_\ZZ(H^{n-1-i}(X^* ,\omega_{X^*}),\ZZ)\to 0
\endaligned \] and for any  closed subset $K^*$ of $X^*$, taking $F=\ZZ_{K^*}$,   exact sequences
\[ \aligned 0\to \Ext^1_\ZZ(H^{n-i}(K^*, \ZZ),\ZZ  )  \to H^i_{K^*} (X^* ,\omega_{X^*}) \to \Hom_\ZZ(H^{n-1-i}(K^*, \ZZ),\ZZ)\to 0
\\ 0\to \Ext^1_\ZZ(H^{n-i}_{K^*} (X^* ,\omega_{X^*}),\ZZ  )  \to H^i  (K^* ,\ZZ) \to \Hom_\ZZ(H^{n-1-i}_{K^*} (X^* ,\omega_{X^*}),\ZZ)\to 0
\endaligned. \]

{\rm\bf (C)} If $K\subseteq X$ is a Cohen--Macaulay closed subset of codimension $d$, one has Gysin  isomorphisms
\[ \aligned H^i_K(X,\omega_X)& = H^{i-d}(K,\omega_K) =\left\{\aligned 0\qquad, &\quad \text{ for }i\neq d
\\   H^{n-d}_\0(K,\ZZ)^*, &\quad \text{ for } i=d, \endaligned\right.
\\ 
\\ H^i_{K^*}(X^*,\omega_{X^*})&=H^{i-d}(K^*,\omega_{K^*})=\left\{ \aligned 0\qquad, & \quad\text{ for } i\neq d, n-1 
\\ \ZZ\qquad, & \quad\text{ for } i=n-1
\\ H^{n-d-1}(K^*,\ZZ)^*, & \quad\text{ for } i=d\endaligned\right.
\endaligned \]  
\end{thm}

\begin{rem}{\rm Notice that Theorem \ref{omega1} applies to any Cohen--Macaulay simplicial complex.}
\end{rem}

\subsection{Cohen--Macaulay spaces} Now let $X$ be a   locally dualizable space, not necessarily local.

\begin{defn} {\rm We say that $X$ is {\em Cohen-Macaulay} if $U_x$ is a Cohen--Macaulay local space for any $x\in X$.}
\end{defn}

\begin{rems}{\rm \begin{enumerate}\item By Remark \ref{CMlocal} this definition is compatible with the previous definition on the local case. 
\item Any open subset of a Cohen--Macaulay space is again Cohen--Macaulay.
\item If $X$ is Cohen--Macaulay, then $d_x=\dim U_x$ is a codimension function.
\item One can also define a Cohen--Macaulay sheaf $F$ as a sheaf such that $F_{\vert U_x}$ is Cohen--Macaulay for any $x\in X$, and then $X$ is Cohen--Macaulay if and only if $\ZZ$ is a Cohen--Macaulay sheaf.
\item If $X$ is connected, Cohen--Macaulay and dualizable, then it has a canonical sheaf $\omega_X$. We leave the reader to reproduce the results on $\omega_X$ of   section \ref{CM-local} to the non-local case.
\item For any locally dualizable space $X$ and any $x\in X$, $C_x$ and $C_x^*:=C_x\negmedspace-\negmedspace\{x\} $ are dualizable and Cohen--Macaulay.
\end{enumerate} }
\end{rems}

%
%

\begin{ex} Let $K^*\overset i\hookrightarrow \PP^r_{\FF_1}$ be a connected and $n$-dimensional projective simplicial complex. Then, $K^*$ is Cohen--Macaulay if and only if
\[ \HExt^i_{\PP^r_{\FF_1}}(\ZZ_{K^*},\ZZ_{\{g\}})=0\text{ for any } i\neq r-n\] where $g$ is the generic point of $\PP^r_{\FF_1}$.
\end{ex}

The next theorems state the behaviour of Cohen--Macaulayness under products and barycentric subdivision.

\begin{thm}\label{CM-product} $X\times Y$ is Cohen--Macaulay if and only if $X$ and $Y$ are Cohen--Macaulay. If moreover  $X$ and $Y$ are dualizable, then $X\times Y$ is also dualizable and  \[\omega_{X\times Y}=\omega_X\boxtimes\omega_Y.\]
\end{thm}

\begin{proof} Regarding Cohen--Macaulayness, one is reduced to the case that $X$ and $Y$ are local, and then the result  follows easily from the equality (Proposition \ref{local-prod}): $D_X^{x_0}\boxtimes D_Y^{y_0}=D_{X\times Y}^{(x_0,y_0)}$. If $X$ and $Y$ are dualizable with dualizing complexes $\Omega_X$ and $\Omega_Y$, it is proved in Proposition \ref{prod-dualizable} that $\Omega_X\boxtimes\Omega_Y$ is a dualizing complex on $X\times Y$, hence the result.
\end{proof}

\begin{thm}\label{barycentric-CM} Let $X$ be a locally dualizable space. Then, $\beta X$ is Cohen--Macaulay if and only if $X$ and $X^{\rm op}$ are Cohen--Macaulay.
\end{thm}

\begin{proof} We shall use the following well known results:

For any finite space $T$, one has \begin{equation}\label{beta-op-invariance} \aligned \widetilde H_i(\beta(T),\ZZ) = \widetilde H_i( T ,\ZZ)\quad &, \quad \dim(\beta(T))=\dim(T)
\\ \widetilde H_i(T,\ZZ)  = \widetilde H_i(T^{\text{\rm op}},\ZZ)\quad &,\quad \dim(T^{\text{\rm op}})=\dim(T).\endaligned \end{equation}

For each $x\in X$, $\widehat x\in X^{\text{\rm op}}$ denotes the same element, but thought in the dual space. Notice that
\[ [U_x]^{\text{\rm op}}=C_{\widehat x},\quad [C_x]^{\text{\rm op}}=U_{\widehat x},\quad [U_x^*]^{\text{\rm op}} =C_{\widehat x}^*,\quad [C_x^*]^{\text{\rm op}} =U_{\widehat x}^*.\]

Assume that $\beta X$ is Cohen--Macaulay and let us prove that $X$ is Cohen--Macaulay. We have to prove that $\widetilde H_i(U_x^*,\ZZ)=0$ for any $x\in X$ and any $i<\dim U_x^*$. Let $\beta=\{x\}\in \beta X$. One has
\begin{equation}\label{U_beta}\aligned U_{\beta} &\overset\sim\longrightarrow \overline \beta (C_x^*) \times \overline\beta (U_x^*)
\\ \{ x_1 <\dots < x <\dots < x_n\}&\mapsto (\{ x_1<\dots <x\}, \{x<\dots <x_n\})\endaligned\end{equation} where $\overline\beta (T)$ is the space obtained by adding a minimal point $0$ to $\beta (T)$. Since $U_\beta$ is Cohen--Macualy, so is $\overline\beta (U_x^*)$ (by Theorem \ref{CM-product});  the Cohen--Macaulayness at $0$ yields
$$\widetilde H_i(\beta (U_x^*),\ZZ)=0 \text{ for } i< \dim \beta (U_x^*)$$ and we conclude by \eqref{beta-op-invariance}.   Finally, since $\beta (X)=\beta(X^{\text{\rm op}})$, we also obtain that $X^{\text{\rm op}}$ is Cohen-Macaulay.

Now assume that $X$ and $X^{\text{\rm op}}$ are Cohen--Macaulay.   Since $\beta X$ is covered by the open subsets $U_\beta$, with $\beta=\{x\}$, it suffices to see that $U_\beta$ is Cohen--Macaulay. By equality \eqref{U_beta} and Theorem \ref{CM-product}, one is reduced to prove that $\overline \beta (C_x^*)$ and $ \overline\beta (U_x^*)$ are Cohen--Macaulay. 

(a)  $ \overline\beta (U_x^*)$ is Cohen--Macaulay. Since $ \overline\beta (U_x^*)=0\sqcup \beta(U_x^*)$ and $\beta(U_x^*)$ is Cohen--Macaulay by induction on $\dim X$ (notice that $U_x^*$ and $[U_x^*]^{\text{\rm op}}=C_{\widehat x}^*$ are Cohen--Macaulay), it only remains to prove the Cohen--Macaulayness at $0$, i.e., that $\widetilde H_i(\beta(U_x^*),\ZZ)=0$ for $i<\dim \beta (U_x^*)$. We conclude by \eqref{beta-op-invariance} because $X$ is Cohen--Macaulay.

(b)  $ \overline\beta (C_x^*)$ is Cohen--Macaulay. Again, by induction (notice that $C_x^*$ and $[C_x^*]^{\text{\rm op}}=U_{\widehat x}^*$ are Cohen--Macaulay), it suffices to prove the Cohen--Macaulayness at $0$, i.e., that $\widetilde H_i(\beta(C_x^*),\ZZ)=0$ for $i<\dim \beta (C_x^*)$. By \eqref{beta-op-invariance}, it suffices to prove the statement for $[C_x^*]^{\text{\rm op}}=U_{\widehat x}^*$.  Since $X^{\text{\rm op}}$ is Cohen--Macaulay, $\widetilde H_i(U_{\widehat x}^*,\ZZ)=0$ for $i<\dim U_{\widehat x}^*$, as wanted.  

\end{proof}

\begin{cor} If $X$ is a simplical complex or a projective simplicial complex, then $X$ is Cohen--Macaulay if and only if $\beta X$ is Cohen--Macaulay.
\end{cor}

\begin{proof} If $X$ is a simplical complex, then $X^{\text{\rm op}}$ is an open subset of $({\A^n_{\FF_1}})^{\text{\rm op}} =\A_{\FF_1}^n$, hence it is Cohen--Macaulay and we conclude by Theorem \ref{barycentric-CM}. Analogously, if $X$ is a projective simplicial complex, then $X^{\text{\rm op}}$ is an open subset of $({\A^n_{\FF_1}}\negmedspace-\negmedspace\0)^{\text{\rm op}}=\A_{\FF_1}^n\negmedspace -\negmedspace\{g\}$, which is also Cohen--Macaulay.
\end{proof}

\subsection{Comparison with Baclawski's Cohen--Macaulayness}

In \cite{B}, a finite poset $P$ is called Cohen--Macaulay if every open interval in $\widehat P$ is a bouquet (this means that these intervals have non zero reduced homology at most in the highest possible dimension), where $\widehat P$ is the poset obtained by adjoining a new par of elements to $P$, written $\widehat \0$, $\widehat\un$  such that $\widehat\0 < x <\widehat\un$ for all $x\in P$.  We shall say that $P$ is CM to refer to  Cohen--Macaulayness in Baclawski's sense. 

A poset is called ACM in \cite{B} if every interval of $\widehat P$ is a bouquet, except possibly for $(x,y)=(\widehat \0,\widehat\un)$.


On simplicial complexes, both definitions agree, and they both agree with Stanley--Reisner Cohen--Macaulayness (i.e., they agree with those simplicial complexes satisfying Reisner's criterion). More precisely: for a simplicial complex $K$, let $K^*$ be the associated projective simplicial complex, i.e., $K^*=K-\{\0\}$. Then
\[ K \text{ is Cohen--Macaulay } \Leftrightarrow K^* \text{ is CM } \Leftrightarrow  K^* \text{ is Stanley-Reisner Cohen-Macaulay}\] and
\[ K^* \text{ is Cohen--Macaulay } \Leftrightarrow K^* \text{ is ACM }. \]

For a general poset, let us see more explicitely the difference between both notions.
\medskip

$(*)$ $X$ is Cohen--Macaulay if and only if it satisfies:\medskip

(a) For every $x<y$ (with $x,y\in X$), the interval $(x,y)$ is a homological sphere.

(b)  For every $x\in X$, $\widetilde H_i(U_x^*,\ZZ)=0$ for $i<\dim U_x^*$.\medskip

$(**)$ $X$ is CM if and only if it satisfies:\medskip

(a') For every $x<y$ (with $x,y\in X$), the interval $(x,y)$ is a bouquet.

(b')=(b)

(c') For every $x\in X$, $\widetilde H_i(C_x^*,\ZZ)=0$ for $i<\dim C_x^*$.

(d') $\widetilde H_i(X,\ZZ)=0$ for $i<\dim X$.\medskip

Finally, $X$ is ACM if it satisfies (a'), (b') and (c').\medskip

Thus, condition (a') is weaker than (a), but CM (resp. ACM) imposes conditions (c') and (d') (resp. (c')) that Cohen--Macaulayness does not. The main difference is that Cohen-Macaulayness is local, while CM is not. Thus, an open subset of a Cohen--Macaulay is Cohen--Macaulay, but this does not hold for CM. In this sense, our definition fits better with that of schemes or commutative algebra. Moreover, our definition agrees better with Stanley--Reisner theory in the following sense. For each simplicial complex $K$, let us denote $S_K$ the associated Stanley--Reisner scheme (that is, the closed subscheme of the affine scheme defined by the Stanley--Reisner ideal). If $K'$ is a subcomplex of $K$, then $S_{K'}$ is a closed subscheme of $S_K$. Then $S_K-S_K'$ is Cohen--Macaulay if and only if $K-K'$ is Cohen--Macaulay. This does not happen with CM (only the if part holds). Also, it is proved in \cite{ST} that one has a continuous map
\[\pi\colon S_K\to K.\] For any open subset $U$ of $K$, $\pi^{-1}(U)$ is an open subscheme of $S_K$, and one has
\[ U\text{ is  Cohen--Macaulay }\Leftrightarrow \pi^{-1}(U) \text{ is a Cohen--Macaulay scheme.}\] This does not hold for CM. If $U$ is CM, then  $\pi^{-1}(U)$ is Cohen--Macaulay, but the converse is not true in general.

Finally, $X$ is CM $\Leftrightarrow$ $\beta X$ is CM while $\beta X$ is Cohen--Macaulay if and only if $X$ and $X^{\text{\rm op}}$ are Cohen--Macaulay. Both statements agree on simplical or projective simplical complexes.

%
%

\end{document}